\documentclass[12pt]{article}
\usepackage[english]{babel}
\usepackage[utf8]{inputenc}
\usepackage{geometry}
\usepackage{yfonts}
\usepackage{tikz}
\usepackage{tikz-cd}
\usepackage{xfrac}
\usepackage[c]{esvect}
\usepackage{setspace}
\usepackage{comment}
\usepackage{graphicx}
\usepackage{indentfirst} 
\usepackage{amsmath,amsfonts,amsthm,amscd,upref,amstext}	
\usepackage[titletoc]{appendix}

\newtheorem{prop}{Proposition}[section]
\newtheorem{theo}[prop]{Theorem}
\newtheorem{cor}[prop]{Corollary}

\newtheorem{defi}[prop]{Definition}

\newtheorem{lem}[prop]{Lemma}
\newtheorem{remark}[prop]{Remark}

\title{Complexity of del Pezzo surfaces with du Val singularities}
\date{}
\author{V.Nadler}
\begin{document}
\maketitle
\doublespacing
\selectlanguage{english}

\abstract{\noindent }
We compute the complexity of del Pezzo surfaces with du Val singularities.
\begin{spacing}{1.0}

\section{Introduction}
The classification of del Pezzo Surfaces with du Val singularities is well-known, cf. \cite{Prokhorov}.  We want to understand "how far" a surface with given singularities is from being toric. By \cite{Shokurov}, the surface is toric if and only if the complexity is zero. Such surfaces are well known. 
In this paper we compute the complexity for del Pezzo surfaces with du Val singularities.

By \emph{a del Pezzo surface} we mean a normal projective Gorenstein surface with ample anti-canonical divisor. 
Let $X$ be a del Pezzo surface.
Let us define \textit{the boundary divisor} as $D=\sum d_iD_i$ with  $0 < d_i \le 1, d_i \in \mathbb{Q}$, where $D_i$ are prime Weil divisors.  Suppose that $K_X+D$ is $\mathbb{Q}$-Cartier, i.e., $\exists n\in \mathbb{N}\colon n(K_X+D)$ is a Cartier divisor.

Recall that the pair $(X,D)$ is called \textit{log-canonical} if for any birational morphism $f: Y \to X$ such that $Y$ is normal we have
\begin{equation*}
K_Y+\widetilde{D} = f^*(K_X+D)+ \sum_{i=1}^N a(X,D,E_i)E_i,    
\end{equation*}
 with $a(X,D,E_i) \ge -1 $ for all $E_i$. Here $\widetilde{D}$ is a strict transform of $D$, $E_i$ are exceptional divisors of $f$.
Denote by $\rho(X)$ the rank of  the Neron-Severi group $\mathrm{NS}(X)$ of $X$. 
The notion of complexity was introduced by V.Shokurov, cf. \cite{Shokurov}.

\begin{defi}
\emph{The complexity}  of $X$ is defined as follows:
\begin{equation}
    \gamma(X) =\mathrm{inf} \lbrace \rho(X)+ \mathrm{dim}X - \sum d_i \colon  K_X+D\equiv 0\rbrace
\end{equation}
where the infimum is taken among all the boundaries $D=\sum d_iD_i$ such that $(X,D)$ is log-canonical. 

\begin{remark}
    If a variety is toric, then the complexity is zero. The converse is true by \cite{McKernan}.
\end{remark}
 
\end{defi}
For convinience we define a similar invariant:
\begin{defi}
Denote the \emph{$\sigma$-invariant} of $X$ as follows:
\begin{equation}
\sigma(X) =\mathrm{\sup_{D}} \lbrace  \sum d_i \colon  K_X+D\equiv 0\rbrace
\end{equation}
where the supremum is taken among all the boundaries $D=\sum d_iD_i$ such that $(X,D)$ is log-canonical.
\end{defi}

It is clear that
\begin{equation}
    \gamma(X)=\mathrm{dim}X+\rho(X)-\sigma(X).
\end{equation}
In this paper, $\sigma(X)$ rather than $\gamma(X)$ is computed, because $\rho(X)$ is always known from the given surface and $\dim X = 2$.

Let $X$ be a del Pezzo surface with du Val singularities. Denote $d=K_X^2$.  It is well known that $d \in \lbrace 1, \dots, 9 \rbrace$.

We prove the following
\begin{theo}

    Whenever $d \ge 7$, we have $\gamma(X)=0$.  
    
    If $5 \le d \le 6$, then $\gamma(X) \in \lbrace 0,1 \rbrace$.

    If $d = 4$, then $\sigma(X) \in \lbrace \frac{3}{2}, \frac{5}{2}, 3, 4 \rbrace$.

    If  $d = 3$, then $\sigma(X) \in \lbrace \frac{4}{3}, 2, \frac{5}{2}, 3 \rbrace$.

    If  $d = 2$, then $\sigma(X) \in \lbrace \frac{7}{6}, \frac{3}{2}, 2 \rbrace$ .

    If $d = 1$, then $\sigma(X) = 1 $ .
\end{theo}

\textbf{Sketch of the proof.} First of all, we calculate the complexity for smooth del Pezzo surfaces in theorem \ref{smooth}. Then for del Pezzo surfaces with du Val singularities we get the upper bound for $\sigma(X)$ in Lemma \ref{kleiman}. Then we use the minimal resolution and reduce the problem to the computations of a similar invariant on smooth weak del Pezzo surfaces (definition \ref{sigma_prime}). Then we use the notion of a cycle which is a special configuration of $(-1)$-curves and $(-2)$-curves on a minimal resolution. If a cycle exists, then to compute the complexity of the initial surface is easy (Lemma \ref{exists}). Smooth weak del Pezzo surfaces with a cycle constitute the majority of all smooth weak del Pezzo surfaces, that is proved in Lemma \ref{induction_lemma}, which provides an induction step. Statements of Section 5  are devoted to the exceptional cases when a cycle exists but the support of the boundary divisor is not snc, so the idea of Lemma \ref{induction_lemma} also works in this case. In Section 6 we work with the case when the cycle does not exist on a minimal resolution and deduce the answer by constructing a special divisor on such resolution.

\section{Complexity of smooth del Pezzo surfaces}
In this section  we compute the complexity of smooth del Pezzo surfaces.
\begin{theo}\label{smooth}
For a smooth del Pezzo surface $X$ of degree $d \ge 7$ we have $\gamma(X) = 0$, and if  $d \le 6$ then $\gamma(X) = 2(6-d)$.
\end{theo}
\begin{remark}
It is well known that all smooth del Pezzo surfaces with $d\ge 6$ are toric, and so $\gamma(X)=0$ in this case. 
\end{remark}
\begin{proof}

Let $D =\sum d_iD_i$ such that $K_X+D\equiv 0$.
For all prime divisors $D_i$  on $X$ we have
    $-K_Y\cdot D_i \ge 1$,
    hence
    \begin{equation*}
     d = K_X^2 = -K_Y\cdot D = -K_Y \cdot \sum d_iD_i \ge \sum d_i.
    \end{equation*}
So, $\sigma(X)\le d$ and hence, 
\begin{equation*}
    \gamma(X)=12-d-\sigma(X)\ge 2(6-d).
\end{equation*}

Note that unless $X \cong \mathbb{P}^1 \times \mathbb{P}^1$, we have
 $\phi \colon X \to \mathbb{P}^2$ is a blow-up of some number of points on $\mathbb{P}^2$.
Denote by $E_i$ the exceptional divisors of a blow-up and by $H$ the strict transform of a general line on $\mathbb{P}^2$. Then
\begin{equation*}
    \mathrm{Pic}(X) \cong \mathbb{Z}[H]\oplus \mathbb{Z}[E_1] \oplus ...\oplus \mathbb{Z}[E_n]
\end{equation*}
where $n = \rho(X) - 1$.
For a prime divisor $D_i$, if $D_i\cdot H=0$, then $D_i = E_j$ for some $1 \le j \le n$.
Then 
\begin{equation*}
    -K_X \equiv D = \sum d_iD_i + \sum e_jE_j
\end{equation*} where prime divisors $D_i$ satisfy $H\cdot D_i \ge 1$.
Then 
\begin{equation*}
    3 = H\cdot D \ge \sum d_i
\end{equation*}
and
\begin{equation*}
    \sum d_i + \sum e_j \le 3 + \rho(X)-1 = \mathrm{dim}X+\rho(X)=12-d.
\end{equation*}
So $\sigma(X) \le 12-d $ and  $\gamma(X)=12-d-\sigma(X) \ge 0$.
Finally, we have $\gamma(X) \ge \mathrm{max}(0,2(6-d))$.

For $X \cong \mathbb{P}^1\times \mathbb{P}^1$, we have
$-K_X \sim 2H$ 
where $H = H_1+H_2$ -- is an ample Cartier divisor, $H_1$ and $H_2$ are divisors of lines isomorphic to $\mathbb{P}^1\times \lbrace p \rbrace $ and $\lbrace q \rbrace\times \mathbb{P}^1$, respectively.
Then 
\begin{equation*}
 8 = K_X^2 = -K_X\cdot D= 2H\cdot \sum d_iD_i \ge 2\sum d_i 
\end{equation*}
and 
\begin{equation*}
    \sum d_i \le 4 = \mathrm{dim}X+\rho(X).
\end{equation*}

Now we show that this estimate is sharp by presenting for each $X$ a divisor $-K_X \equiv D$ which has maximal sum of coefficients.
For $d \ge 6$, one can pick the sum of toric divisors  on~$X$. 
For $d \le 5$ and for  $X \cong Bl_{p_1,p_2, ... p_n}\mathbb{P}^2$, $d = 9 - n$, $3 \le  n \le 7$ pick 
\begin{equation*}
D = \frac{d}{N(d)}\sum E_i
\end{equation*}
where the sum is taken among the set of $(-1)$-curves $E_i$ and  $N(d)$ is a number of $(-1)$-curves on a surface.
Recall that 
\begin{equation*}
    N(6)=6,\quad N(5)=10,\quad N(4)=16,\quad N(3)=27, \quad N(2)=56.
\end{equation*}
Corresponding pair $(X,D)$ is log-canonical, cf. \cite{Cheltsov}.
When $X \cong Bl_{p_1,...p_8}(\mathbb{P}^2)$, take $D = C$ where $C$ is a smooth elliptic curve in $|-K_X|$.
\end{proof}





\section{Geometry of weak del Pezzo surfaces}
In this section we recall the geometrical properties of smooth weak del Pezzo surfaces which will be applied for the computations for du Val del Pezzo surfaces. We will use the following notation.

 A smooth projective surface $X$ has du Val singularities if the exceptional divisor of the minimal resolution $f\colon Y \to X$ consists of $(-2)$-curves, i.e., smooth rational curves $A_j$ which satisfy $A_j^2=-2$. The dual graph of the resolution is a Dynkin diagram of type $\mathbb{A,D,E}$. Recall  that a smooth projective surface $Y$ is a \textit{weak del Pezzo surface} if $-K_Y$ is nef and big. 
 Suppose that $f\colon Y \to X$ is the minimal resolution of a du Val del Pezzo surface $X$. Then $Y$ is a weak del Pezzo surface.
Classification of smooth weak del Pezzo surfaces is given in \cite{Watanabe}.
The $\textit{degree}$ of a del Pezzo surface $X$ is an integer number $d=K_X^2$ and $d\in \lbrace 1,\dots,9 \rbrace$. Because $f$ is $\textit{crepant}$, i.e., $K_Y=f^*K_X$, we have $K_Y^2=K_X^2=d$.
Unless $Y \cong \mathbb{F}_2$ (this holds if and only if  $Y$ is the minimal resolution of the cone over a conic), $Y$ is a blow-up of $\mathbb{P}^2$ in some number of points in almost general position \cite{Watanabe}, i.e, 
there is a sequence of blow-ups
\begin{equation*}
Y = Y_n \to Y_{n-1} \to \dotsb \to Y_1 \to Y_0 \cong \mathbb{P}^2,    
\end{equation*}
 where the point of $i$-th blow-up $p_i$ does not belong to a $(-2)$-curve of $Y_{i-1}$ and $n=9-d$.

We introduce some notation which is used in the paper.

\begin{defi}\label{D(Y)}
    
Denote by $D(Y)$ the finite set of all curves on $Y$ with negative self-intersection.
Note that $D(Y)=A(Y)\cup E(Y)$ where $A(Y)$ and $E(Y)$ are sets of $(-2)$-curves $A_i$ on $Y$ and $(-1)$-curves $E_i$ on $Y$, respectively. Recall that  $f$ contracts $A_i \in A(Y)$ to singular points on $X$. 
\end{defi}

Note that $X$ is $\mathbb{Q}$-factorial \cite{Watanabe}, that is, any Weil divisor is a $\mathbb{Q}$-Cartier divisor.
For any Weil divisor $D$ on $X$ we have $f^*(D) = f_*^{-1}(D) + \sum a_iA_i$ where $a_i \in \mathbb{Q}$ and $f_*^{-1}(D)$ is a strict transform of $D$.
Because $f$ is crepant, the pair $(X,D)$ is log-canonical if and only if $(Y,f^*(D))$ is log-canonical.

\begin{defi}\label{sigma_prime}
    
For a weak del Pezzo surface $Y$ denote 
\begin{equation}
\sigma'(Y) = \mathrm{\sup_{D}}\lbrace \sum d_i \colon  K_Y+D \equiv 0 \rbrace ,
\end{equation}
where the supremum is taken among such $D$ that 
\begin{equation*}
    D=\sum  d_iD_i + \sum a_jA_j, \quad D_i^2\ge -1, \quad A_i \in A(Y), \quad 0<d_i,a_j\le 1, \quad (Y,D) \text{ is log-canonical}.
\end{equation*}

\end{defi}
\begin{remark}
    If $f\colon Y \to X$ is the minimal resolution of $X$, then $\sigma(X)=\sigma'(Y)$.
\end{remark}

\begin{prop}\label{MoriCone}
    When $-K_Y^2 = d \le 7$, the Mori Cone $\mathrm{\overline{NE(Y)}}$ is spanned by $(-1)$-curves and $(-2)$-curves on $Y$. 
\end{prop}

\begin{proof}
    From the Cone theorem \cite{Kollar-Mori} follows that
    \begin{equation*}
        \overline{NE(Y)} = \overline{NE(Y)_{K_Y=0}} + \sum_{i}\textbf{R}_{\ge 0}[E_i], \quad E_i\in E(Y).
    \end{equation*} Note that if $-K_Y\cdot C =0$, then $C^2 < 0$ and $supp(C)$ consists of $(-2)$-curves. Hence, $\overline{NE(Y)_{K_Y=0}}$ is spanned by curves from $A(Y)$, which is a finite set.
\end{proof}











\begin{lem} \label{lem34}
Suppose that $d\ge 3$.

    \textbf{i)} Suppose that an effective divisor $L$ on $Y$  satisfies $L^2=0$ and $-K_Y\cdot L =2$. Then
    \begin{equation*}
     h^0(Y,\mathcal{O}_{Y}(L))=2.
    \end{equation*}


    
    \textbf{ii)} Suppose a $(-1)$-curve $E$ satisfies $E\cdot L = 0$ and $L$ satisfies the conditions of i). Then 
   \begin{equation*}
    h^0(Y, \mathcal{O}_{Y}(L-E))=1.
   \end{equation*}

    \textbf{iii)} Suppose that $L$ from i) has a fixed component $F$ and a movable component $M$ such that $F = L - M$. 
    Then $M^2=0$ and $-K_Y\cdot L = 2$ and hence, i) holds for the effective divisor $M$. More over, $\mathrm{supp}(F)$ consists of $(-2)$-curves on $Y$ and $M$ is base point free.
    
    \textbf{iv)} Suppose that $L$ is from i) and $L$ is base point free. Then a general member of $L$ is a curve $C$ such that $C \cong \mathbb{P}^1$ and $L$ admits a decomposition
    \begin{equation*}
        L \sim \sum_{i}e_jE_i + \sum_{j}a_jA_j, \quad 0 < e_i,a_j \in \mathbb{Z}, \quad E_i \in E(Y), \quad A_j \in A(Y).
    \end{equation*} Possible decompositions are described in the Table 1.

\end{lem}

\begin{proof}

    \textbf{i)} Note that $-K_Y\cdot (K_Y+L) = 2-d < 0$ and, because $-K_Y$ is nef, we have $h^0(Y,\mathcal{O}(K_Y+L))=0$. Also 
    \begin{equation*}
        h^2(Y,\mathcal{O}(K_Y+L))=h^0(Y,\mathcal{O}(-L))=0.
    \end{equation*}
    By Riemann-Roch theorem, we have
    \begin{equation*}
        h^1(Y,\mathcal{O}(K_Y+L))= -(1 + \frac{1}{2}L\cdot(K_Y+L))=0.
    \end{equation*}
Then from the  short exact sequence
    \begin{equation*}
        0 \longrightarrow \mathcal{O}_Y(K_Y+L) \longrightarrow \mathcal{O}_Y(L) \longrightarrow    \mathcal{O}_C(L_{|C}) \longrightarrow 0,
    \end{equation*}
    where $C \sim -K_Y$ is a non-singular elliptic curve, follows that the sequence 
\begin{equation*}
        0 \longrightarrow H^0(Y,\mathcal{O}_Y(K_Y+L)) \longrightarrow H^0(Y,\mathcal{O}_Y(L)) \longrightarrow    H^0(C,\mathcal{O}_C(L_{|C})) \longrightarrow 0
    \end{equation*}
is also exact. Then
\begin{equation*}
    h^0(Y,\mathcal{O}_Y(L)) = h^0(\mathcal{O}_C(L_{|C})) = L\cdot C
\end{equation*},
where the last equality follows from Riemann-Roch for curves and $g(C)=1$. Hence, $h^0(Y,\mathcal{O}_Y(L)) = 2$.

    \textbf{ii)} Because $L\cdot E = 0$, we have $(L-E)^2 = -1$. We have by Serre duality
    \begin{equation*}
        h^2(Y,\mathcal{O}_Y(L-E)) =h^0(Y,\mathcal{O}_Y(K_Y+E-L))=0.
    \end{equation*}
    By Riemann-Roch we deduce that 
    \begin{equation*}
        h^0(Y,\mathcal{O}_Y(L-E)) \ge 1+\frac{1}{2}(L-E)(L-E-K_Y)=1.
    \end{equation*} Suppose that $h^0(Y,\mathcal{O}_Y(L-E)) = 2$. We prove that $L-E$ does not have a movable component. Indeed, suppose that there is a divisor $F+M \in |L-E|$ such that $M$ is movable.
    We have  $-K_Y \cdot M \neq 0$ because otherwise $\mathrm{supp}(M)$ consists of $(-2)$-curves. Then
    $1 \le -K_Y\cdot M \le -K_Y \cdot L = 1$ because $-K_Y$ is nef. Because $d \ge 2$, we have $ -K_Y\cdot(M+K_Y) = 1 - d < 0$, so $h^0(Y,\mathcal{O}_Y(M+K_Y)) = 0$.
    By Riemann-Roch theorem, we have
    \begin{equation*}
       - h^1(Y,\mathcal{O}_Y(M+K_Y))  = 1 + \frac{1}{2}(M+K_Y)\cdot M = 1 + \frac{1}{2}(M^2 - 1).
    \end{equation*}
    Then $M^2 < 0$ which is a contradiction. So, $L-E$ does not have movable components. Then $h^0(Y,\mathcal{O}_Y(L-E))$ coincides with the number of connected components in $L-E$, which equals to $1$. Indeed, only one connected component $C$ from $\mathrm{supp}(L-E)$ satisfies $-K_Y \cdot C = 1$ and $C^2 = -1$, the other components $D$ may consist only of $(-2)$-curves and
    \begin{equation*}
    -1 = (L-E)^2 = (C+D)^2 = C^2 + D^2 \le -1,    
    \end{equation*}
     so $D^2 = 0$ and $D = 0$, because the intersection form is negative-definite on $(-2)$-curves on $Y$. 
    
    \textbf{iii)} Note that $-K_Y \cdot M > 0 $ because if $-K_Y \cdot M = 0$, then $\mathrm{supp}(M)$ consists of $(-2)$-curves and $M^2 < 0$. We have $2 = h^0(Y,\mathcal{O}_{Y}(L))\ge h^0(Y,\mathcal{O}_{Y}(M))\ge 2$ because $M$ is a movable divisor. Hence, $h^0(Y,\mathcal{O}_{Y}(M)) = 2$. From Riemann-Roch theorem follows that
    \begin{equation*}
         2 - h^1(Y,\mathcal{O}_{Y}(M)) = 1 + \frac{1}{2}(M(M-K_Y)) \ge 2,
    \end{equation*}
    because $-K_Y\cdot M >0$ and $M^2 \ge 0$. Then $M^2 -K_Y\cdot M =2$. Suppose that $-K_Y\cdot M = 1$. Then $h^0(Y,\mathcal{O}_Y(K_Y+M))=0$ because $-K_Y\cdot(K_Y+M) = 1 - d < 0$. From that follows that 
    \begin{equation*}
        -h^1(Y,\mathcal{O}_Y(K_Y+M)) = 1 + \frac{1}{2}M\cdot(K_Y+M) = 1,
    \end{equation*}
    which is a contradiction. Hence, $-K_Y\cdot M = 2$ and $M^2=0$. We deduce that $-K_Y\cdot F = 0$ and, hence, $\mathrm{supp}(F)$ consists of $(-2)$-curves.

    \textbf{iv)} By proposition \ref{MoriCone}, one can write 
         $L = \sum e_iE_i + \sum a_jA_j$, where $e_i$ and $a_j$ are nonnegative rational numbers.
         Then $2 = -K_Y\cdot L = \sum e_i$ and at least one $(-1)$-curve $E$ is included in the decomposition. Then $L\cdot E = 0$, because $L$ is base point free. By ii), one can write
         \begin{equation*}
         L-E = \sum \Tilde{e_i}E_i + \sum \Tilde{a_j}A_j, \quad 0 \ge \tilde{e_i}, \Tilde{a_j} \in \mathbb{Z}.    
         \end{equation*}
          So, $L$ admits an integer decomposition.
         
         Now consider the contraction $\pi \colon Y \to Z$, which contracts $E$ to a point $p$. Then general members of $L$ and $\pi_*(L)$ are isomorphic because $L \cdot E = 0$, $(\pi_*L)^2=0$ and $-K_Z\cdot \pi_*(L)=2$.
         One can arrange a sequence of contractions to $Z = Bl_p{\mathbb{P}^2}\cong \mathbb{F}_1$ and deduce that if
         $-K_Z\cdot L = 2$ and $L^2=0$ and $L$ is base point free, then $L$ is a fiber and $L \cong \mathbb{P}^1$. 

         One can get all the possible integer decompositions by blowing up a fiber
         on $Bl_p{\mathbb{P}^2}$ and taking reducible members of the pullback.

\begin{center}
    \begin{tabular}{ | l | l | p{4cm}|}
    \hline
    Type & Dual graph & Decomposition of $L$ \\ \hline
    $1$
    &
    \begin{minipage}{0.3\textwidth}
\begin{center}
     
    \begin{tikzpicture}
        \node[circle,draw=black,scale=0.5](1) at (-3,0) {};
        \node[]() at (-3,0.5){$1$};
        \node[circle,draw=black,scale=0.5](2) at (-2,0) {};
        \node[]() at (-2,0.5){$1$};
    \path (1) edge (2);
    \end{tikzpicture}

\end{center}

\begin{center}

    \begin{tikzpicture}
        \node[circle,draw=black,scale=0.5](1) at (-3,0) {};
        \node[]() at (-3,0.5){$1$};
        \node[circle,fill=black,scale=0.5](2) at (-2,0) {};

        \node[]() at (-2,0.5){$1$};
        \node[](3) at (-1,0) {$\dots$};
        \node[circle,fill=black,scale=0.5](4) at (0,0) {};
        \node[]() at (0,0.5){$1$};
        \node[circle,draw=black,scale=0.5](5) at (1,0) {};
        \node[]() at (1,0.5){$1$};
    \path (1) edge (2);
    \path (2) edge (3);
    \path (3) edge (4);
    \path (4) edge (5);
    \end{tikzpicture}
\end{center}
    \end{minipage}

    & $L\sim E_1+A_1+\dots +A_l+E_2; l \ge 0$ \\ \hline
    $2$
    &
    \begin{minipage}{0.3\textwidth}\begin{center}
    \begin{tikzpicture}
        \node[circle, fill=black,scale=0.5](1) at (-1,0) {};
        \node[circle, draw=black,scale=0.5](2) at (0,0) {};
        \node[circle, fill=black,scale=0.5](3) at (1,0) {};
        \node[]() at (-1,0.5) {$1$};
        \node[]() at (0,0.5) {$2$};
        \node[]() at (1,0.5) {$1$};
    \path (1) edge (2);
    \path (2) edge (3);
    \end{tikzpicture}
\end{center}
    \end{minipage}
    & $L\sim 2E_1+A_1+A_2$ \\ \hline
    $3$
    &
    \begin{minipage}{0.3\textwidth}

    \begin{center}
      \begin{tikzpicture}
        \node[circle, draw=black,scale=0.5](1) at (-1,0) {};
        \node[]() at (-1,0.5){$2$};
        \node[circle, fill=black,scale=0.5](2) at (0,0) {};
        \node[]() at (0,0.5){$2$};
        \node[](3) at (1,0) {$\dots$};
        \node[circle,fill=black,scale=0.5](4) at (2,0) {};
        \node[]() at (2,0.5){$2$};
        \node[circle,fill=black,scale=0.5](5) at (3,0.5) {};
        \node[]() at (3,1){$1$};
        \node[circle,fill=black,scale=0.5](6) at (3,-0.5) {};
        \node[]() at (3,0){$1$};
    \path (1) edge (2);
    \path (2) edge (3);
    \path (3) edge (4);
    \path (4) edge (5);
    \path (4) edge (6);
    \end{tikzpicture}
    \end{center}
    \end{minipage}
    & 
    $L\sim 2E_1+2A_1+2A_2+\dots+2A_l+A_{l+1}+A_{l+2}; l \ge 1$ \\ \hline
\end{tabular}

\end{center}

\end{proof}

\begin{remark}
    For $d = 2$ the previous lemma does not hold. For example, take $L = A + C$, where $A$ is a $(-2)$-curve and $C \sim -K_Y$ is a non-singular elliptic curve. Then $L^2=0$ and $-K_Y\cdot L =2$, but 
    \begin{equation*}
        h^0(Y,\mathcal{O}_{Y}(L)) \ge  h^0(Y,\mathcal{O}_{Y}(-K_Y)) = d + 1 = 3.
    \end{equation*}
\end{remark}

\begin{lem}\label{genus_zero}
    Suppose that $C$ is a smooth curve with $g(C)=0$ and an effective divisor $D$ such that $D\cdot C \ge 1$. Then for any point $p \in C$ there exists $L \in |D|$ such that $p \in L$.
\end{lem}

\begin{proof}
    Let $\hat{D} = D_{|C}$. We have 
    \begin{equation*}
        \chi(\hat{D}) = D\cdot C + \chi(\mathcal{O}_C) = C\cdot D + 1,
    \end{equation*} so $h^0(C,\mathcal{O}_C(\hat{D})) \ge 2$, the statement then follows from the fact that all points on $\mathbb{P}^1 \cong C$ are linearly equivalent.
\end{proof}

The next lemma clarifies the structure of $(-1)$-curves on $Y$.
\begin{lem}\label{intersection_minus_one_curves}
    Suppose that $Y$ is a weak del Pezzo surface with $d = K_Y^2 $. Then, when $d \ge 3$, for any two  $(-1)$-curves $E_1, E_2 \in Exc(Y)$ we have $E_1\cdot E_2 \le 1$. If $d = 2$, then $E_1\cdot E_2 \le 2$ and if $E_1\cdot E_2 = 2$, then $E_1+E_2 \sim -K_Y$.
\end{lem}

\begin{proof}
   \textbf{Step 1.} Suppose that $d \ge 3$. Because $-K_Y\cdot (E_1+E_2+K_Y) = 2-d < 0$ and $-K_Y$ is nef, we have 
     $h^0(Y,\mathcal{O}_{Y}(E_1+E_2+K_Y))=0$. Then by Riemann-Roch theorem, we have
     \begin{equation*}
         -h^1(Y,\mathcal{O}_{Y}(E_1+E_2+K_Y))= 1 + \frac{1}{2}(E_1+E_2)\cdot (E_1+E_2+K_Y) = 1 +(E_1\cdot E_2 - 2).
     \end{equation*}
     So, 
     \begin{equation*}
         E_1\cdot E_2 = 1-h^1(Y,\mathcal{O}_{Y}(E_1+E_2+K_Y)).
     \end{equation*}

    \textbf{Step 2.} Suppose that $d=2$ and $E_1\cdot E_2 \ge 2$. Riemann-Roch theorem gives, as in the previous case,
    \begin{equation*}
        h^0(Y,\mathcal{O}_{Y}(E_1+E_2+K_Y))-h^1(Y,\mathcal{O}_{Y}(E_1+E_2+K_Y))= E_1\cdot E_2 - 1 \ge 1.
    \end{equation*}
    So, $h^0(Y,\mathcal{O}_{Y}(E_1+E_2+K_Y)) \ge 1$.
    Because $-K_Y\cdot (E_1+E_2+K_Y) = 2-d = 0$, we have the decomposition $E_1+E_2+K_Y = \sum a_iA_i$, where $A_i$ are $(-2)$-curves on $Y$. But 
    \begin{equation*}
        0 \ge (E_1+E_2+K_Y)^2 = 2(E_1\cdot E_2 - 2),
    \end{equation*} so  $E_1\cdot E_2 \le 2 $ and, once  $E_1\cdot E_2 = 2$, we have  $a_i = 0$ and $E_1+E_2+K_Y \sim 0$, because the self-intersection form on $(-2)$-curves is negative definite. 
\end{proof}

\begin{remark}\label{mult_remark}
    Suppose that $\pi \colon Z= Bl_pY\to Y$ is a blow-up of $Y$ in $p$, where $Y$ is a weak del Pezzo surface and $L \subset Y$ is a Weil divisor. Suppose that $\hat{L} \subset Z$ is a strict transform of $L$ and $E$ is an exceptional divisor of $\pi$. Then $\mathrm{mult}_pL = E\cdot \hat{L}$. This follows from the equations of a blow-up.
\end{remark}

\section {Computations for del Pezzo surfaces with du Val singularities with $D(Y)$ snc}

In this section we prove  Lemma \ref{exists} and Lemma \ref{induction_lemma}.
Everywhere below we denote by $f\colon Y \to X$ the minimal resolution of $X$. Recall that $Y$ is a smooth weak del Pezzo surface.
\begin{lem}\label{kleiman}
    $\sigma'(Y) \le d$.
\end{lem}
\begin{proof}

We write 
\begin{equation*}
    -K_Y \equiv D = \sum d_iD_i + \sum a_jA_j, \quad A_j \in D(Y), \quad A_j^2=-2,\quad D_i \quad \text{are prime divisors}, \quad D_i^2>-2.
\end{equation*} We have
$-K_Y\cdot D_i \ge 1$, $-K_Y\cdot A_j=0$.
Hence, 
    \begin{equation*}
        d = K_Y^2 = -K_Y\cdot(\sum d_iD_i + \sum a_jA_j) \ge \sum d_i.
    \end{equation*}
From the Remark 3.2 follows that $\sigma(X) = \sigma'(Y) \le d$.
\end{proof}

\begin{lem}
    For $d=1$ we have $\sigma'(Y)=1$.
\end{lem}
\begin{proof}
    It is known from \cite{Watanabe} that there is a non-singular elliptic curve $C$ in $|-K_Y|$, so in the previous lemma the equality holds.
    \end{proof}

    Everywhere below we assume that $d\ge 2$.


\begin{defi} \label{cycle_def}
   We say that the set of curves $I$  is a \emph{cycle} in $D(Y)$ if $I = \lbrace C_1, \dots C_m \rbrace $ with $m\ge 3$, $C_i \in D(Y)$ and satisfy 
   \begin{equation*}
       C_1\cdot C_2=C_2\cdot C_3 = \dots = C_{m-1}\cdot C_m = C_m\cdot C_1 = 1.
   \end{equation*} Also we call $I = \lbrace C_1,C_2 \rbrace \subset D(Y)$ a cycle if $C_1\cdot C_2 = 2$.
\end{defi}  

\begin{remark}
    The pair $(Y,C)$, where $C = \sum_{C_i \in I}C_i$, is not supposed to be log-canonical. For example, if $d=3$, one can take a cycle of $\lbrace C_1,C_2,C_3\rbrace$, where $C_i$ are lines which intersect in a point $p$. 
\end{remark}
\begin{defi}
   By the \emph{content} $c(I)$ of a cycle $I$  we mean the number of $(-1)$-curves in I.
\end{defi}
\begin{prop}
    $(\sum_{C_i \in I} C_i)^2 = -K_Y\cdot \sum_{C_i \in I} C_i=c(I)$.
\end{prop}
\begin{proof}
    It follows from the definition of the cycle.
\end{proof}
\begin{lem}\label{exists}
     $c(I)\ge d$. Moreover, if $c(I)=d$, then

     \begin{equation*}
      -K_Y \sim \mathrm{\sum_{C_i \in I} C_i}.
     \end{equation*}
\end{lem}
\begin{proof}
Denote $C = \sum_{i=1}^{d+r} E_i + \sum A_j $, where $r=c(I)-d \in \mathbb{Z}$, $E_i \in E(Y)$, $A_j \in A(Y)$. Note that $c(I)>0$ because $Y$ is a minimal resolution of a du Val del Pezzo surface and a connected set of $(-2)$-curves from $A(Y)$ is a Dynkin diagram of type $\mathbb{A},\mathbb{D},\mathbb{E}$, which is a tree. 
By Riemann-Roch, 
\begin{equation*}
    h^0(K_Y+C) \ge 1+ \frac{1}{2}(K_Y+C)\cdot C = 1
\end{equation*}
Hence, there exists an effective divisor in $|K_Y+C|$.
Because $-K_Y$ is nef, we have 
\begin{equation*}
0 \le -K_Y \cdot (K_Y+C) = c(I)-d. 
\end{equation*}
Suppose that  $c(I)=d$. Then  $-K_Y \cdot (K_Y+C) = 0$ and we have
\begin{equation*}
    K_Y+C \sim  \sum a_jA_j, \quad 0 < a_j \in \mathbb{Z}.
\end{equation*}
Then $0 = (K_Y+C)^2 = (\sum a_jA_j)^2$.
We deduce that all $a_i=0$, because the intersection form is negative-definite on subgroup in $\mathrm{Pic}(Y)$, generated by $(-2)$-curves. Then  $K_Y+C \sim 0$.
\end{proof}

By the previous lemma, if it is possible to find a cycle of content $d$ in $D(Y)$ which is snc, then $\sigma'(Y)=d$.
The next lemma provides an induction step from surfaces of degree $d$ to surfaces of degree $d-1$.

\begin{lem} \label{induction_lemma}
    Suppose that $K_Y^2=d \ge 3$. Let $\phi \colon Z = Bl_p(Y) \to Y$, such that $p$ does not belong to any curve from $A(Y)$. Note that $Z$ is a weak del Pezzo surface as well. Suppose that $supp(D(Y))$ and $supp(D(Z))$ are snc. Let $I$ be a cycle in $D(Y)$, $c(I)=d$, so $-K_Y\sim C$.
    Suppose that $p$ is not a point of intersection of two $(-1)$-curves from $E(Y)\backslash I$. Then there is a cycle $J$ on $Z$ such that $c(J) = d-1$ and $\sigma'(Z)=d-1$.
\end{lem}
\begin{proof}

We consider three cases according to the position of $p$.

    \textbf{Case 1.} 
    If $p$ lies on a $(-1)$-curve from $I$, then  $-K_Z \equiv D  = \phi_{*}^{-1}(C)+\delta E$. Here $\delta=0$ if $p$ lies exactly on one $(-1)$-curve from $I$ and $\delta=1$ if $p$ is the intersection of two $(-1)$-curves from $I$.
    If $\delta=0$, then the cycle $J$ consists of all the strict transforms of the curves from $D$. If $\delta=1$, then the cycle $J$ consists of all the strict transforms of the curves from $D$ and the exceptional divisor $E$ of the blow-up.  Note that $c(J)=d-1$, because the number of $(-1)$-curves in the cycle decreases of $n - \delta = 1$, where $n$ is the number of $(-1)$-curves which contain the point $p$.
    
    \textbf{Case 2.} 
    If $ p \notin D(Y)$ and $p$ lies on a $(-1)$-curve $E$ which is not a component of $C$, then $E\cdot C = 1$ and there is the only one curve $F$ from $I$ which intersects $E$ in a point $q$. Now take  $D =\sum \limits_{C_i \in I} C_i + F$, where the sum is taken among such $C_i$ that $\mathrm{supp}(D)$ is connected and $-K_Y\cdot D = 2$, i.e, $D$ contains exactly two $(-1)$-curves. In fact, $supp(D)$ is a circuit from a cycle $I$, i.e., the sequence of curves $C_1 \dots C_k$ such that 
    \begin{equation*}
    C_1 \cdot C_2 = C_2 \cdot C_3 = \dots = C_{k-1} \cdot C_k = 1
    \end{equation*}
    and the rest intersection numbers $C_i \cdot C_j = 0$ for $i \neq j$. 
    Note that  $D^2=0$ and $-K_Y\cdot D = 2$. By Lemma \ref{lem34}, 
    \begin{equation*}
        h^0(D) = 2,
    \end{equation*}
    Because $D\cdot E = 1$, it is possible to choose a curve $L \in |D|$ such that $p\in supp(L)$. If $L$ is reducible, then there is a $(-1)$-curve $\tilde{E} \subset \mathrm{supp}(L)$ which passes through $p$, which is a contradiction. So $L$ is an irreducible divisor, hence, it is a rational curve with $L^2=0$.
    Then $J$ consists of the strict transforms of the curves from $I$ which are not included to $D$ and $\phi_{*}^{-1}(L)$, which is a $(-1)$-curve. One can check that $J$ is a cycle.

    \textbf{Case 3}.  
    If $p \notin D(Y)$, then choose any $D$ made from $I$ in the way above. The structure of $J$ is the same that in the last case.
\end{proof}
\begin{remark} \label{weak_contraction}
    If $\pi \colon Y \to Y'$ is the contraction of a $(-1)$-curve on $Y$ which is weak del Pezzo surface, then $Y'$ is also a weak del Pezzo surface.  Indeed, for $C \subset Y'$ we have 
    \begin{equation*}
        -K_{Y'}\cdot C = -K_Y \cdot \pi^*(C) = -K_Y \cdot (\Tilde{C}+eE) = -K_Y\cdot C +e \ge 0
    \end{equation*} where $e=mult_{\pi(E)}(C)$. Hence, $-K_{Y'}$ is nef. Because $K_{Y'}^2=K_{Y}^2+1 > 0$ and $-K_{Y'}$ is nef, $-K_{Y'}$ is big.
\end{remark}

We need one more statement about the case when there exists a cycle $I \in D(Y)$, $c(I)=d$ and $(-1)$-curves $C_1,C_2$ such that  $p \in C_1 \cap C_2$, $C_1,C_2 \notin I$. 

Obviously it is possible to include $C_1, C_2$ into 2 cycles $I_1, I_2$ but they might have $c(I_j) > d$.

\begin{lem} \label{coffin4}
    Suppose that $d=4$, and there is a cycle $I \in D(Y)$ such that $c(I)=4$, $C_1,C_2 \notin I$ are $(-1)$-curves which intersect in $p\notin D(Y)$. Then it is possible to include $C_1$ and $C_2$ into a cycle $J$ such that $c(J)=4$.
\end{lem}

\begin{proof}
We prove the statement by contradiction. Suppose that the desired cycle does not exist.

\textbf{Step 1.}
    Denote $I=\lbrace E_1, \dots, E_2, \dots, E_3, \dots, E_4, \dots \rbrace$ where there may be $(-2)$-curves instead of dots.    
    Because $\sum_{i=1}^4 E_i + \dots \sim -K_Y$, $C_1$ and $C_2$ intersect one of the curves of $I$ in  points $q_1$ and $q_2$, correspondingly.
    If $C_1$ intersects a $(-2)$-curve from $I$, one can build the cycle $J$ as follows: start from this $(-2)$-curve and go clockwise or anti-clockwise in $I$ until a curve from $I$ which intersects $C_2$ is reached.
    Let us now consider that $(C_1,E_1)=1$. If $(C_2,E_3)=0$, analogously, start from $E_1$ and go clockwise or anti-clockwise in $I$ until a curve from $I$ which intersects $C_2$ is reached. 
    Let us finally consider that $(C_1,E_1)=(C_2,E_3)=1$. Now the content of the cycles built in the way above is $5$.

\textbf{Step 2.} 
Let us prove that $(E_1,E_2)=0$, i.e., $I=\lbrace E_1, A_1,\dots, A_k, E_2, \dots \rbrace$.
    Suppose that $(E_1,E_2)=1$.
Here lemma \ref{lem34} is applied. Denote by $L$ the curve such that $L^2=0$ and $E_2+ \tilde{A}_1 + \dots +\tilde{A}_l + E_3 \sim L$. Then $(L,C_1)=0$ and, because $(L,C_2)=1$, there exist $(-2)$-curves $A_1,\dots A_k$ and a $(-1)$-curve $E$ such that $L \sim C_1+A_1+\dots +A_k + E$. Note that $E\cdot E_2=E\cdot E_3 = 0 $, $E\neq E_1$ because $(E_1,E_2)=1$ and $E \neq E_4$ because otherwise the cycle $ K = \lbrace C_1, \dots , E = E_4, \dots E_1 \rbrace$ satisfies $c(K) = 3$, which is a contradiction. Then $E$ intersects either $E_4$  or one of $(-2)$-curves. In both cases $C_1$ can be included into a new cycle of content equal to $4$. 

\textbf{Step 3.} 
From step 2 immediately follows that
\begin{equation*}
    (E_1,E_2)=(E_2,E_3)=(E_3,E_4)=(E_4,E_1)=0
\end{equation*} by symmetry.
Consider a blow-up $\psi \colon \widetilde{Y}=Bl_p(Y) \to Y$. Notice that $\psi^{-1}(E_1)$ is a $(-1)$-curve and it intersects three $(-2)$-curves. Then contract $\psi^{-1}(E_1)$ and get a weak del Pezzo surface of degree 4 which has three $(-1)$-curves $D_1,D_2,D_3$ which intersect in a point, which leads to a contradiction because then $I=\lbrace D_1,D_2,D_3 \rbrace$ is a cycle of content $3$.







\end{proof}
\begin{lem} \label{cyclic_d}
    
Suppose that $d \ge 3$ and there exists a cycle $I \subset D(Y)$. Then there exists a cycle $J \subset D(Y)$ such that $c(J)=d$.
\end{lem}

\begin{proof}
    If $c(I)=d$, then $J=I$. Suppose that $c(I) > d$. Denote $C = \sum \limits_{C_i \in I} C_i $.
    
    \textbf{Step 1.} 
    We prove that there exists a $(-1)$-curve $E$ such that $E \notin I$ and $E\cdot C = 0$ . 
    Indeed, we have $(C+K_Y)^2 = d-c(I) < 0$ and 
    \begin{equation*}
        h^0(Y,\mathcal{O}_Y(C+K_Y)) \ge h^1(Y,\mathcal{O}_Y(C+K_Y)) + 1 \ge 1.
    \end{equation*} Hence, $C+K_Y$ has a nonempty fixed component $F$ and a movable component $M$ such that $C+K_Y = F+M$. Then $F\cdot (C+K_Y) < 0$, otherwise $(C+K_Y)^2 \ge 0$, which is a contradiction.
    Note that $C$ is nef. Indeed, for all curves $C_i$ from $I$ we have $C_i\cdot C \ge 0$.
    Note that $F$ consists of $(-1)$-curves $E_i$ and $(-2)$-curves $A_j$, we have $F = \sum e_iE_i + \sum a_jA_j$, where $0 \le e_i, a_j$.
    We have $A_j\cdot (C+K_Y) = A_j\cdot C \ge 0$, hence, there exists $E_i$ such that $E_i\cdot (C+K_Y) < 0$. Because $C$ is nef, we have
    \begin{equation*}
        0 > E_i\cdot(C+K_Y) = E_i\cdot C - 1,
    \end{equation*}
    from that follows that $E_i\cdot C = 0$. Step 1 is finished.

    \textbf{Step 2.}
    Contract $E_i$ which is found in the step 1 and apply the remark \ref{weak_contraction} to contract $Y \to Z$, apply the step $1$ again for a weak del Pezzo surface $Z$ with $K_Z^2 = K_Y^2 + 1$. The content of a cycle does  not change within the contraction, so step $1$ is applied until the degree of a surface becomes equal to $c(I)$. Then the result follows from lemma \ref{induction_lemma} and lemma \ref{coffin4}.

\end{proof}

\begin{cor}\label{result3}
    When $d \ge 3 $, if $D(Y)$ is snc and if $D(Y)$ contains a cycle, there exists $1$-complement $D=\sum D_i \equiv -K_Y$.
\end{cor}

\begin{remark}\label{bad_blowup}
    Lemma \ref{cyclic_d} does not hold for $d=2$, there is the only one weak del Pezzo surface of degree $2$ with the all cycles of the content at least $3$.
\end{remark}

\begin{lem}  \label{coffin3}  
Suppose  that $d=3$,  $D(Y)$ is snc, there is a cycle $I$ and $c(I)=3$, also $p=C_1\cap C_2$ where $C_i$ are $(-1)$-curves and $C_i\not\subset I$. Suppose that $C_1$ and $C_2$ cannot be included into cycles with content $3$.  Let $\phi \colon Z=Bl_p(Y) \to Y$ be the blow-up at a point $p$.
Then $\sigma'(Z)=2$ and $Z$ is the unique surface from the remark \ref{bad_blowup}.
\end{lem}

\begin{proof}
By $...$ we mean that there are possibly $(-2)$-curves instead of dots.
Suppose that 
\begin{equation*}
    I=\lbrace E_1, A_1, \dots, A_k , E_2, A'_1, \dots, A'_l, E_3, A''_1, \dots, A''_m, A, A'''_1, \dots, A'''_n \rbrace , \quad 0 \le k,l,m,n.
\end{equation*}
If the index is zero, the corresponding $(-2)$-curves are absent.
Assume that $(C_1,A)=(C_2,E_2)=1$, otherwise apply the argument of step 1 from lemma \ref{coffin4} to build the cycle.
Denote $L \sim E_1 + \sum_{i=1}^k A_i + E_2$. We have $C_1\cdot L=0$, $L^2=0$ and $-K_Y\cdot L = 2$, then apply ii) of  lemma \ref{lem34}  for the divisor $L-C_1$. We have $L \sim C_1 + \dots + E$, where $E$ is a $(-1)$-curve.
$E \neq E_1$, because $L-C_1$ is the unique divisor. Also $E \neq C_2$, because $L\cdot C_2 = 1$. Analogously, $E \neq E_2$ because of the uniqueness of the decomposition of $L-E_2$, and $E \neq C_1$ because $L\cdot C_2 = 1$.
If $E \notin I$, then $E$ intersects $I$ in $E_3$ or $A'_i$ and then $J = \lbrace C_1, \dots, E_3, \dots \rbrace$ is a cycle which satisfies $c(J)=3$ and includes $C_3$, which is a contradiction.
The last option is when $E=E_3$. Then 
\begin{equation*}
    L = C_1+A+A''_1+\dots+A''_m + E_3,
\end{equation*} 
hence, because $L\cdot E_1 = L\cdot E_2 =0$, we have
    $A\cdot E_1 = E_3 \cdot E_2 = 0$.
Analogously, repeating the proof for a divisor $\tilde{L} = E_2 + \dots + E_3$, we have
    $A\cdot E_3 = E_1 \cdot E_2 = 0$.
We have proved that $1 \le k,l,m,n$. 
Let us prove that $1 = k,l,m,n$.
Take the divisor 
\begin{equation*}
    L = A_k + 2E_2 + A'_1 \tag{$\star$} \label{first},
\end{equation*} 
which satisfies the conditions of lemma \ref{lem34}. We have $L\cdot C_2 = 2$ and $L\cdot C_1 = 0$ and $L$ is base point free by iv) of lemma \ref{lem34}.
Let us prove that 
\begin{equation*}
    L \sim 2C_1 + 2A + A''_m + A'''_1. \tag{$\star \star$} \label{second}
\end{equation*} 
Indeed,
\begin{equation*}
    L\cdot A = L\cdot A''_m = L\cdot A'''_1 = 0,
\end{equation*} which is seen from \eqref{first}. Then $A, A''_m, A'''_1 \subset \mathrm{supp}(L-C_1)$ and one can apply iv) of lemma \ref{lem34} for $L$ to get the coefficients of \eqref{second}.
From that immediately follows that $A'''_1\cdot E_1 = A_k\cdot E_1$ and $A''_m\cdot E_3 = A'_1\cdot E_3$.
To finish the proof, let us look at the divisor $\hat{L}=C_2+E_2$. We have $\hat{L}\cdot E_1 = \hat{L}\cdot E_3 = 0$ and one can apply ii) of lemma \ref{lem34} to show that
\begin{equation*}
\hat{L}\sim E_3 + A''_1 + \dots + A''_m + A + A'''_1 + \dots + A'''_n + E_1.
\end{equation*}
Then $E_1\cdot A_k = \hat{L}\cdot A_k = 1$ and $E_3\cdot A'_1 = \hat{L}\cdot A'_1 =1$. This proves that $k=l=m=n=1$ and finally,
\begin{equation*}
    I = \lbrace E_1, A_1, E_2, A'_1, E_3, A''_1, A, A'''_1 \rbrace.
\end{equation*}
By the proof above, $D(Y)$ is uniquely defined, i.e., $D(Y) = I \cup C_1\cup C_2$. Indeed, one can contract $C_1,E_1$ and $E_3$ and get the smooth del Pezzo surface $S$ of degree $6$.
To build the $2$-complement on $Z = Bl_p(Y)$, take 
\begin{equation*}
    D = C_1+C_2+E_2+A+\frac{1}{2}(A_1+A'_1+A''_1+A'''_1 ) \equiv -K_Y.
\end{equation*}
Note then $mult_p(D)=2$. Blow up $Y$ at $p$ and get $\Tilde{D}=f^{-1}(D) + E \equiv -K_Z$, where $E$ is the divisor of the blow up. We have the $2$-complement on $Z$ and $\sigma'(Z)=2$. Note that $Z$ contains only the cycles of content $3$.
The uniquiness of such a surface is finally proved 
That finishes the proof.
\end{proof}

The result of this paragraph can be expressed in the next
\begin{cor}
    Suppose that $D(Y)$ is snc and there exists a cycle $I \subset D(Y)$. Then $\sigma'(Y)=d$ and for all $p\in Y$ such that $p$ does not belong to any $(-2)$-curve, we have
    $\sigma'(Bl_p(Y))=d-1$.
\end{cor}

\begin{proof}
    Apply corollary \ref{result3} and lemma \ref{coffin3}.
\end{proof}

\section{The case when $D(Y)$ is not snc}
In this section we work with the cases when $D(Y)$ from the definition \ref{D(Y)} is not snc.

Suppose that a cycle $J$ with $c(J)=d$ on $Y$ is found, i.e, $D_J=\sum_{E_i \in J} E_i + \sum_{A_j \in J} A_j \sim -K_Y$. The problem arises when $(Y,D)$ is not log-canonical, so $D$ cannot be taken as a  divisor for which the estimation is sharp.
The purpose of this paragraph is to show that, nevertheless, $\sigma'(Y)=d$ in this case.
Obviously, $D$ is not snc only if $D$ is either a sum of three lines which pass through a point, or a sum of two curves which are tangent in a point, and $(Y,D)$ is log-canonical if and only if $D$ is snc.

\begin{lem} \label{notsnc3}
    If $d=3$ and $D$, which is defined above, is not snc, then $\sigma'(Y)=3$.
\end{lem}

\begin{proof}
Denote $D=\sum_{i=1}^3 \mathbb{E}_i$, where $\mathbb{E}_i$ are $(-1)$-curves from $D(Y)$ and $q = \cap_{i=1}^3 \mathbb{E}_i$. 
Assume the contraction $\phi \colon Y \to Z$, where $\phi(\mathbb{E}_3)=p$ is a point in $Z$ .
We have $L_1$,$L_2$, which are images of $\mathbb{E}_1$ and $\mathbb{E}_2$ under the contraction, so $L_i^2=0$ and $-K_Z \cdot L_i = 2$. Then, by lemma \ref{lem34}, one can replace $L_1$ with corresponding combination of $(-1)$-curves and $(-2)$-curves. We get $-K_Z\equiv D = L_2 + \sum e_iE_i + \sum a_jA_j$, when, from the same lemma, $e_i,a_j \le 2$. Now blow up $D$ at $p\in L_2$ and get $-K_Y \equiv D^* = \phi^*(D)=E_2+\sum e_i\phi^{-1}(E_i)+\sum a_j\phi^{-1}(A_j)$.
Finally, take $C = \frac{1}{2}(E_1+E_2+E_3+D^*)$.
Obviously, $-K_Y\equiv C$ and $(Y,C)$ is log-canonical. Hence, $\sigma'(Y)=3$, which finishes the proof.
\end{proof}

\begin{remark}
    In lemma \ref{notsnc3} there is the unique surface $S$ such that $D(S)$ contains three lines which intersect in a point and such that $\sigma'(S)=3$ and there is no $1$-complement on $S$. This surface can be built from the example of the remark \ref{bad_blowup} by contracting the unique $(-1)$-curve which intersects three $(-2)$-curves.
\end{remark}

\begin{lem}\label{notsnc2}
    Suppose that $d=2$ and there are two $(-1)$-curves and one $(-2)$-curve on $Y$ which intersect in a point or two $(-1)$-curves which are tangent in a point. Then $\sigma'(Y)=2$.
\end{lem}

\begin{proof}
    Suppose the first case, i.e., $\mathbb{E}_1$, $\mathbb{E}_2$ and $\mathbb{A}$ intersect in $q\in Y$.
Assume the contraction $\phi \colon Y \to Z$, $\phi(\mathbb{E}_1)=p$ is a point in $Z$.
As in lemma \ref{notsnc3}, we have $-K_Z \equiv L+E$, where $L$ and $E$ are images of $\mathbb{E}_2$ and $\mathbb{A}$ under the contraction. Then the proof is exactly the same as in the previous lemma, namely, represent $L$ in a form $D=\sum e_iE_i +\sum a_jA_j$ and then blow up at $p\in E$ the divisor $E+D$, then take $C=\frac{1}{2}(\phi^*(E+D)+\mathbb{E}_1+ \mathbb{E}_2+\mathbb{A})$. It is clear that $-K_Y\equiv C$ and $(Y,C)$ is log-canonical.
\newline

Suppose the second case, i.e., $\mathbb{E}_1$ and $\mathbb{E}_2$ are tangent in $q\in Y$.
Consider the contraction $\phi \colon Y \to Z$, $\phi(\mathbb{E}_1)=p$ is a point in  $Z$ .    
Then $\phi_*(\mathbb{E}_2)\equiv -K_Z$.
Note that, as $\mathbb{E}_1+\mathbb{E}_2\equiv -K_Z$, there is no $(-2)$-curves intersecting $\mathbb{E}_1$ and, hence, no $(-1)$-curves which pass through $p\in Z$. 
If  either $D(Z)$ is not a tree or $supp(D(Z))$ is not snc, then find a cycle $I$ such that $c(I)=3$ and apply lemma \ref{induction_lemma} for the case $p \notin D(Z)$. If $I$ consists of  three $(-1)$-curves which intersect in a point, then apply lemma \ref{notsnc3}.
If $D(Y)$ is a tree, see theorem \ref{Treetheorem} below and notice that in case 
 $Y$ is isomorphic to $Y_{3.1},Y_{3.2},Y_{3.3}$, one can find a curve $L$ such that $L^2=0$ and $(-K_Y,L)=2$ and $p \in L$. Then apply lemma \ref{lem34},  write 
 \begin{equation*}
 -K_Y \equiv D = L + \sum e_iE_i + \sum a_jA_j    
 \end{equation*}
  where integer coefficients $e_i, a_j $ satisfy $e_i, a_j \le 1$. Then blow up $D$ and get 1-complement on $Y$.
If $Y=Y_{3,4}$, then take
\begin{equation*}
    C=L+E+2A_1+3A_2+4A_3+2A_4+3A_5+2A_6
\end{equation*}
and
\begin{equation*}
    D=\frac{1}{4}(\phi^{-1}(C))+\frac{3}{4}(\mathbb{E}_1+\mathbb{E}_2).
\end{equation*}
 The pair $(Y,D)$ is log-canonical because $lct(Y,\mathbb{E}_1+\mathbb{E}_2)=\frac{3}{4}$. 
\end{proof}

In the next paragraph we show that when $p \notin D(Y_{3,4})$,the least complement index for the boundary on $Y=Bl_p(Y_{3,4})$ is $4$.

To finish this paragraph, we need one more statement.

\begin{lem}
    Suppose that $d=3$ and $p=\cap_{i=1}^3 C_i$, $C_i\in D(Y)$ are $(-1)$-curves. Than for any point $q$ that does not lie on a $(-2)$-curve and for $\phi \colon Z=Bl_q(Y) \to Y$ we have $\sigma'(Z)=2$.
\end{lem}

\begin{proof}

    Consider $\pi \colon Z=Bl_q(Y) \to Y$.
    
    \textbf{Case 1} : suppose that $q$ is not a point of intersection of three $(-1)$-curves. Suppose that $q\not\in \cup_{i=1}C_i$. Because $C_1,C_2,C_3$ together form a cycle, apply lemma \ref{induction_lemma}, if $q$ belongs to two curves $E_1, E_2$, apply lemma \ref{coffin3}.

    \textbf{Case 2}: suppose that $q\in C_1$ and $q \not = p$. Take $L_1 \sim C_2+C_3$ such that $q \in L_1$. Note that $-K_Y \equiv L_1 + C_1$. Note that because $L_1\cdot C_1=2$, $L_1$ and $C_1$ may be tangent at $q$. 
    As the boundary on $Z$ take 
    \begin{equation*}
        \Tilde{D}=\frac{1}{2}\pi^{-1}(L_1+2C_1+C_2+C_3).
    \end{equation*} 
    Pair $(Z,\Tilde{D})$ is log-canonical. 

    \textbf{Case 3}: $q=p$. Two options are possible: either there is one more cycle $J$ on $Y$ with $c(J)=3$ or $I=\lbrace C_1,C_2,C_3\rbrace$ is the unique cycle with the minimal content.
    Notice that if $J$ exists,  $J$ cannot contain two curves from $I$ because then the third one intersects $J$ in $p$ with multiplicity $2$, which is impossible.
 
    \textbf{Case 3.1}: If one of the curves of $I$ is in $J$ (say, $C_1$), then blow up $J$ and get the cycle of length $2$ on $Y$ which contains at least three curves and they do not intersect in a point, hence, $(Z,\pi^{-1}C_J)$ is log-canonical, where $C_J$ is a sum of all curves in $J$.
    
    \textbf{Case 3.2}: If the second cycle $J$ does not contain any of $C_i$, then curves $C_i$ intersect $(-1)$-curves of $J$. Indeed, $(C_i,J)=1$ and because $C_1+C_2+C_3\equiv -K_Y$, then $(-2)$-curves do not intersect $I$. Suppose that $E_1,E_2,E_3$ are $(-1)$-curves of $J$ and $C_i\cdot E_i=1$ and $C_i\cdot E_j = 0$ if $i\neq j$. Note that a divisor $L=C_1+C_2$ satisfies $L\cdot E_3 = 0$, hence, $L$ admits a decomposition of the form $L = E_3 + \dots + E$ by lemma \ref{lem34}. If $E \neq E_3$, then $E_3,\dots,E,C_3$ form a cycle with the minimal content which contains $C_3$, this is discussed in the case 3.2. So suppose that $E=E_3$ and $L = 2E_3 + \dots$. Then, because $L\cdot E_1 = L\cdot E_2 = 1$, we have $E_3\cdot E_1 = E_3\cdot E_2 = 0$. Analogously, taking $L = C_1+C_3$, we have in this case $E_1\cdot E_2 = 0$. Finally the structure of $J$ is as follows:
    \begin{equation*}
        J = \lbrace E_1, A_1,\dots,A_k,E_2,A'_1,\dots,A'_l,E_3,A''_1,\dots,A''_m\rbrace, 1 \le k,l,m
    \end{equation*}
    One can take 
    \begin{equation*}
        D = C_1 + E_1 + \frac{1}{2}(C_2+C_3+\sum_{i=1}^k A_i + \sum_{j=1}^l A''_j)
    \end{equation*} and get $\hat{D} = \pi^{-1}(D)+E$, where $E$ is the exceptional divisor of the blow up. One can check that $K_Z+\hat{D} \equiv 0$ and $(Z,\hat{D})$ is log-canonical, so $\sigma'(Z)=2$.

   \textbf{Case 3.3} : $I$ is the unique cycle on $Y$ such that $c(I)=3$.
    There is at least one more $(-1)$-curve $E_1\in D(Y)$, otherwise $D(Y)$ contains just $C_1,C_2,C_3$ which is impossible because $\rho(Y)=10-d=7$.
Suppose that $(E_1,C_1)=1$. Obviously, $p \notin E_1$. Then, because $(C_2+C_3,E_1)=0$, by lemma \ref{lem34}, there exists $L_1 \sim C_2+C_3$ such that $E_1 \subset supp(L_1)$.One can write the  decomposition 
\begin{equation*}
    L_1=e_1E_1+ \sum d_iD_i,\quad 1 \le e_1 \le 2, \quad 0\le d_i\le 2 \text{for} \quad e_1,d_i \in \mathbb{Z}, \quad D_i \in D(Y).
\end{equation*} 
    Take 
    \begin{equation*}
        D=\frac{1}{2}(2C_1+C_2+C_3+L_1).
    \end{equation*} Take $\hat{D}= \pi^{-1}(D)$. Then $(Z,\hat{D})$ is log-canonical, so $\hat{D}$ is the 2-complement on $Z$ and $K_Z+\hat{D}\equiv 0$, which means that $\sigma'(Z)=2$.
\end{proof}

\section{The case when $D(Y)$ is a tree}
In this section we work with the case when $D(Y)$ from the definition \ref{D(Y)} is a tree.

\begin{defi}\label{denominator}
    $n(Y)$, or the $\mathrm{denominator}$ of $D(Y)$ is 
    \begin{equation*}
        \min_{D}\lbrace \max_{\lbrace d_i\rbrace}(d_i) \colon  -K_Y \equiv D = \sum d_iD_i, D_i \in D(Y), 0\le d_i \rbrace .
    \end{equation*}
\end{defi}
\begin{lem}
Let $e = \max(\sum e_j)$ be the maximum.                                
The following inequality holds:
\begin{equation*}
    \frac{1}{2}(d+e) \ge \sigma'(Y) \ge 1 + \frac{d-1}{n(Y)}
\end{equation*}
    Here the maximum is taken by all 
    \begin{equation*}
        D = \sum d_iD_i + \sum e_jE_j + \sum a_kA_k, D \equiv -K_Y, 0 \le d_i,e_j,a_k,
    \end{equation*}
    where prime divisors $D_i, E_j, A_k $ satisfy
    \begin{equation*}
    D_i^2\ge 0, E_j^2=-1, A_k^2=-2.
    \end{equation*}
\end{lem}
\begin{proof}
    To prove the first inequality note that for 
    \begin{equation*}
        -K_Y \equiv D = \sum d_iD_i + \sum e_jE_j + \sum a_kA_k.
    \end{equation*}
    Here $D_i,E_j,A_k$ are prime divisors,
    \begin{equation*}
        D_i^2 \ge 0, \quad E_j^2=-1, \quad A_k^2=-2.
    \end{equation*}
    We have
    \begin{equation*}
        d=(-K_Y)^2= \sum d_i(-K_Y,D_i)+\sum e_j \ge 2\sum d_i + \sum e_j = 2(\sum d_i + \sum e_j) - \sum e_j
    \end{equation*}
    because $(-K_Y,D_i)\ge 2$.
    To show the second inequality take $C = \frac{1}{n(Y)}D + (1-\frac{1}{n(Y)})\mathcal{C}$, where $\mathcal{C} \in |-K_Y|$ is a non-singular elliptic curve, 
    $D$ is from the previous definition, i.e., $\mathrm{sup}(D)$ consists of $(-1)$-curves and $(-2)$-curves from $D(Y)$. 
\end{proof}

Now we build all the weak del Pezzo surfaces with $D(Y)$ a tree and $d \ge 2$.
Begin with $d=7$, here are $2$ surfaces. Since $d\le 7$, $NE(Y)$ is spanned by all $(-1)$-curves and $(-2)$-curves.
The following algorithm allows to build $D(Y)$ by induction.
Suppose that $D(Y)$ is given. let $Z=Bl_p(Y)$. 
To define the structure of $D(Z)$ provided that $D(Z)$ is a tree, we need to understand which prime divisors $L\subset Y$, such that $L^2=-1 + r^2$ and $-K_Y\cdot L = 1 + r$, pass through $p$. Here $r = \mathrm{mult}_pL$. By lemma \ref{intersection_minus_one_curves} and remark \ref{mult_remark}, $r = \hat{L}\cdot E \le 2$, where $\hat{L}$ is a strict transform of $L$ and $\hat{L}^2 = L^2 - r^2 =-1$ and $-K_Z \cdot \hat{L} = -K_Y \cdot L - r = 1$, so $\hat{L}$ is a $(-1)$-curve.If $r = 2$, we have $d=3$ and there is a cycle $I = \lbrace \hat{L},E \rbrace$ on $Z$, so if there exists a cubic on $Y$ which has a singular point $p$, $D(Z)$ contains a cycle. So in the case $D(Y)$ and $D(Z)$ are trees, new $(-1)$-curves on $D(Z)$ appear only from curves $L$ with $L^2=0$ and $-K_Y\cdot L =2$.
If $p$ lies on a $(-1)$-curve $E_1 \in D(Y)$, then $\mathrm{mult}_pE_1 = 1$ and $\hat{E}_1^2 = -2$, so $\hat{E}_1$ is a $(-2)$-curve on $Z$. 
If $p\in E \subset D(Y)$, then find the combinations of curves from $D(Y)$ from lemma \ref{lem34} which give a class $L$ such that $L^2=0$, $-K_Y\cdot L = 2$ and  $(L,E)\ge 1$, if these combinations exist. Then, by lemma \ref{genus_zero}, $p \in C \in |L|$  for some $C$.
If $p \notin D(Y)$, then, by lemma \ref{genus_zero}, for each $L$ defined above there is a unique curve $C \in |L|$ that $p\in C$.



    
    
    
    
    

\begin{theo} \label{Treetheorem}

For each type of surfaces from the previous lemma the values of $\gamma(X)$ are described in the tables 2 and 3.

\begin{table}
    \begin{tabular}{ | l | l | l | l | p{2.5cm}|}
    \hline
    Surface & $D_K$, Decomposition of $-K_Y$ & Slave divisor & $\gamma(X)$ &  $D\equiv -K_Y$ \\ \hline
    $Y_{7,1}$
    & 
        \begin{minipage}{0.35\textwidth}

    \begin{center}
    \begin{tikzpicture}
        \node[circle, draw=black,scale=0.5](1) at (-1,0) {};
        \node[circle, draw=black,scale=0.5](2) at (0,0) {};
        \node[circle, draw=black,scale=0.5](3) at (1,0) {};
        \node[]() at (-1,0.5) {$2$};
        \node[]() at (0,0.5) {$3$};
        \node[]() at (1,0.5) {$2$};
    \path (1) edge (2);
    \path (2) edge (3);
    \end{tikzpicture}
\end{center}
    \end{minipage}

    &
    
        \begin{minipage}{0.35\textwidth}
    \begin{center}
    \begin{tikzpicture}
        \node[circle, draw=black,scale=0.5](1) at (-1,0) {};
        \node[circle, draw=black,scale=0.5](2) at (0,0) {};
        \node[circle, draw=black,scale=0.5](3) at (1,0) {};
        \node[]() at (-1,0.5) {$1$};
        \node[]() at (0,0.5) {$2$};
        \node[]() at (1,0.5) {$1$};
    \path (1) edge (2);
    \path (2) edge (3);
    \end{tikzpicture}
\end{center}
    \end{minipage}
     & 0 &$L_1+L_2+\sum E_i$\\ \hline
    $Y_{7,2}$
    & 
    
        \begin{minipage}{0.35\textwidth}
    \begin{center}
    \begin{tikzpicture}
        \node[circle, draw=black,scale=0.5](1) at (-1,0) {};
        \node[circle, draw=black,scale=0.5](2) at (0,0) {};
        \node[circle, fill=black,scale=0.5](3) at (1,0) {};
        \node[]() at (-1,0.5) {$3$};
        \node[]() at (0,0.5) {$4$};
        \node[]() at (1,0.5) {$2$};
    \path (1) edge (2);
    \path (2) edge (3);
    \end{tikzpicture}
\end{center}
    \end{minipage}
    & 
    
        \begin{minipage}{0.35\textwidth}

    \begin{center}
    \begin{tikzpicture}
        \node[circle, draw=black,scale=0.5](1) at (-1,0) {};
        \node[circle, draw=black,scale=0.5](2) at (0,0) {};
        \node[circle, fill=black,scale=0.5](3) at (1,0) {};
        \node[]() at (-1,0.5) {$1$};
        \node[]() at (0,0.5) {$2$};
        \node[]() at (1,0.5) {$1$};
    \path (1) edge (2);
    \path (2) edge (3);
    \end{tikzpicture}
\end{center}
    \end{minipage}
    & 0&{\footnotesize$L+H+E_1+E_2+A$}\\ \hline
    
    $Y_{6,1}$
    & 
    
        \begin{minipage}{0.35\textwidth}
    \begin{center}
    \begin{tikzpicture}
        \node[circle, draw=black,scale=0.5](1) at (-1,0) {};
        \node[circle, fill=black,scale=0.5](2) at (0,0) {};
        \node[circle, draw=black,scale=0.5](3) at (1,0) {};
        \node[circle, fill=black,scale=0.5](4) at (2,0) {};
        \node[]() at (-1,0.5) {$2$};
        \node[]() at (0,0.5) {$3$};
        \node[]() at (1,0.5) {$4$};
        \node[]() at (2,0.5) {$2$};
    \path (1) edge (2);
    \path (2) edge (3);
    \path (3) edge (4);
    \end{tikzpicture}
\end{center}
    \end{minipage}
    &
    
        \begin{minipage}{0.35\textwidth}
        \begin{center}
    \begin{tikzpicture}
        \node[circle, draw=black,scale=0.5](1) at (-1,0) {};
        \node[circle, fill=black,scale=0.5](2) at (0,0) {};
        \node[circle, draw=black,scale=0.5](3) at (1,0) {};
        \node[circle, fill=black,scale=0.5](4) at (2,0) {};
        \node[]() at (-1,0.5) {$1$};
        \node[]() at (0,0.5) {$2$};
        \node[]() at (1,0.5) {$2$};
        \node[]() at (2,0.5) {$1$};
    \path (1) edge (2);
    \path (2) edge (3);
    \path (3) edge (4);
    \end{tikzpicture}
\end{center}
    \end{minipage}
    & 0 &{\footnotesize$L_1+L_2+\sum (E_i+A_i)$} \\ \hline 
    $Y_{6,2}$ 
    & 
        \begin{minipage}{0.35\textwidth}
        \begin{center}
    \begin{tikzpicture}
        \node[circle, draw=black,scale=0.5](1) at (-1,0) {};
        \node[circle, draw=black,scale=0.5](2) at (0,0) {};
        \node[circle, fill=black,scale=0.5](3) at (1,0) {};
        \node[circle, draw=black,scale=0.5](4) at (2,0) {};
        \node[circle, draw=black,scale=0.5](5) at (3,0) {};
        \node[]() at (-1,0.5) {$1$};
        \node[]() at (0,0.5) {$2$};
        \node[]() at (1,0.5) {$2$};
        \node[]() at (2,0.5) {$2$};
        \node[]() at (3,0.5) {$1$};
        
    \path (1) edge (2);
    \path (2) edge (3);
    \path (3) edge (4);
    \path (4) edge (5);
    \end{tikzpicture}
\end{center}
    \end{minipage}
    & 
    
        \begin{minipage}{0.35\textwidth}
            \begin{center}
    \begin{tikzpicture}
        \node[circle, draw=black,scale=0.5](1) at (-1,0) {};
        \node[circle, draw=black,scale=0.5](2) at (0,0) {};
        \node[circle, fill=black,scale=0.5](3) at (1,0) {};
        \node[circle, draw=black,scale=0.5](4) at (2,0) {};
        \node[circle, draw=black,scale=0.5](5) at (3,0) {};
        \node[]() at (-1,0.5) {$1$};
        \node[]() at (0,0.5) {$2$};
        \node[]() at (1,0.5) {$2$};
        \node[]() at (2,0.5) {$1$};
        \node[]() at (3,0.5) {$0$};
        
    \path (1) edge (2);
    \path (2) edge (3);
    \path (3) edge (4);
    \path (4) edge (5);
    \end{tikzpicture}
\end{center}
    \end{minipage}
    & 0&$L+A+\sum E_i$ \\ \hline
    $Y_{6,3}$ 
    & 
    
        \begin{minipage}{0.35\textwidth}
            \begin{center}
    \begin{tikzpicture}
        \node[circle, draw=black,scale=0.5](1) at (-1,0) {};
        \node[circle, fill=black,scale=0.5](2) at (0,0) {};
        \node[circle, fill=black,scale=0.5](3) at (1,0.5) {};
        \node[circle, draw=black,scale=0.5](4) at (1,-0.5) {};
        \node[]() at (-1,0.5) {$3$};
        \node[]() at (0,0.5) {$4$};
        \node[]() at (1,1) {$2$};
        \node[]() at (1,0) {$3$};
    \path (1) edge (2);
    \path (2) edge (3);
    \path (2) edge (4);
    \end{tikzpicture}
\end{center}
    \end{minipage}
    & 
    
        \begin{minipage}{0.35\textwidth}
                \begin{center}
    \begin{tikzpicture}
        \node[circle, draw=black,scale=0.5](1) at (-1,0) {};
        \node[circle, fill=black,scale=0.5](2) at (0,0) {};
        \node[circle, fill=black,scale=0.5](3) at (1,0.5) {};
        \node[circle, draw=black,scale=0.5](4) at (1,-0.5) {};
        \node[]() at (-1,0.5) {$1$};
        \node[]() at (0,0.5) {$2$};
        \node[]() at (1,1) {$1$};
        \node[]() at (1,0) {$1$};
    \path (1) edge (2);
    \path (2) edge (3);
    \path (2) edge (4);
    \end{tikzpicture}
\end{center}
    \end{minipage}
    & 1&{\footnotesize$L+H+A_1+A_2+E_1$}\\ \hline
    $Y_{6,4}$ 
    &  
    
        \begin{minipage}{0.35\textwidth}
        \begin{center}
    \begin{tikzpicture}
        \node[circle, fill=black,scale=0.5](1) at (-1,0) {};
        \node[circle, draw=black,scale=0.5](2) at (0,0) {};
        \node[circle, fill=black,scale=0.5](3) at (1,0) {};
        \node[circle, fill=black,scale=0.5](4) at (2,0) {};
        \node[]() at (-1,0.5) {$3$};
        \node[]() at (0,0.5) {$6$};
        \node[]() at (1,0.5) {$4$};
        \node[]() at (2,0.5) {$2$};
    \path (1) edge (2);
    \path (2) edge (3);
    \path (3) edge (4);
    \end{tikzpicture}
\end{center}
    \end{minipage}
    & 
        \begin{minipage}{0.35\textwidth}
        \begin{center}
    \begin{tikzpicture}
        \node[circle, fill=black,scale=0.5](1) at (-1,0) {};
        \node[circle, draw=black,scale=0.5](2) at (0,0) {};
        \node[circle, fill=black,scale=0.5](3) at (1,0) {};
        \node[circle, fill=black,scale=0.5](4) at (2,0) {};
        \node[]() at (-1,0.5) {$1$};
        \node[]() at (0,0.5) {$2$};
        \node[]() at (1,0.5) {$2$};
        \node[]() at (2,0.5) {$1$};
    \path (1) edge (2);
    \path (2) edge (3);
    \path (3) edge (4);
    \end{tikzpicture}
\end{center}
    \end{minipage}
    & 0&$L+H+E+\sum A_i$ \\ \hline    
    $Y_{6,5}$ 
    & 
        \begin{minipage}{0.35\textwidth}
                    \begin{center}
    \begin{tikzpicture}
        \node[circle, draw=black,scale=0.5](1) at (-1,0) {};
        \node[circle, fill=black,scale=0.5](2) at (0,0) {};
        \node[circle, draw=black,scale=0.5](3) at (1,0.5) {};
        \node[circle, draw=black,scale=0.5](4) at (1,-0.5) {};
        \node[]() at (-1,0.5) {$2$};
        \node[]() at (0,0.5) {$3$};
        \node[]() at (1,1) {$2$};
        \node[]() at (1,0) {$2$};
    \path (1) edge (2);
    \path (2) edge (3);
    \path (2) edge (4);
    \end{tikzpicture}
\end{center}
    \end{minipage}
    & 
        \begin{minipage}{0.35\textwidth}
                    \begin{center}
    \begin{tikzpicture}
        \node[circle, draw=black,scale=0.5](1) at (-1,0) {};
        \node[circle, fill=black,scale=0.5](2) at (0,0) {};
        \node[circle, draw=black,scale=0.5](3) at (1,0.5) {};
        \node[circle, draw=black,scale=0.5](4) at (1,-0.5) {};
        \node[]() at (-1,0.5) {$1$};
        \node[]() at (0,0.5) {$2$};
        \node[]() at (1,1) {$1$};
        \node[]() at (1,0) {$1$};
    \path (1) edge (2);
    \path (2) edge (3);
    \path (2) edge (4);
    \end{tikzpicture}
\end{center}
    \end{minipage}
    & 1&{\footnotesize$L_1+L_2+E_1+E_2+A$} \\ \hline
    $Y_{5,1}$ 
    & 
        \begin{minipage}{0.35\textwidth}
         \begin{center}
    \begin{tikzpicture}
        \node[circle, draw=black,scale=0.5](0) at (-2,0) {};
        \node[circle, fill=black,scale=0.5](1) at (-1,0) {};
        \node[circle, fill=black,scale=0.5](2) at (0,0) {};
        \node[circle, draw=black,scale=0.5](3) at (1,0.5) {};
        \node[circle, fill=black,scale=0.5](4) at (1,-0.5) {};
        
        \node[]() at (-2,0.5) {$2$};
        \node[]() at (-1,0.5) {$3$};
        \node[]() at (0,0.5) {$4$};
        \node[]() at (1,1) {$3$};
        \node[]() at (1,0) {$2$};
        
    \path (0) edge (1);
    \path (1) edge (2);
    \path (2) edge (3);
    \path (2) edge (4);
    \end{tikzpicture}
\end{center}
    \end{minipage}
    & 
        \begin{minipage}{0.35\textwidth}
             \begin{center}
    \begin{tikzpicture}
        \node[circle, draw=black,scale=0.5](0) at (-2,0) {};
        \node[circle, fill=black,scale=0.5](1) at (-1,0) {};
        \node[circle, fill=black,scale=0.5](2) at (0,0) {};
        \node[circle, draw=black,scale=0.5](3) at (1,0.5) {};
        \node[circle, fill=black,scale=0.5](4) at (1,-0.5) {};
        
        \node[]() at (-2,0.5) {$1$};
        \node[]() at (-1,0.5) {$2$};
        \node[]() at (0,0.5) {$2$};
        \node[]() at (1,1) {$1$};
        \node[]() at (1,0) {$1$};
        
    \path (0) edge (1);
    \path (1) edge (2);
    \path (2) edge (3);
    \path (2) edge (4);
    \end{tikzpicture}
\end{center}
    \end{minipage}
    & 1&$L_1+L_2+\sum_{i=1}^3A_i$\\ \hline
    $Y_{5,2}$ 
    &  
        \begin{minipage}{0.35\textwidth}
            \begin{center}
    \begin{tikzpicture}
        \node[circle, fill=black,scale=0.5](1) at (-1,0) {};
        \node[circle, draw=black,scale=0.5](2) at (0,0) {};
        \node[circle, fill=black,scale=0.5](3) at (1,0) {};
        \node[circle, fill=black,scale=0.5](4) at (2,0) {};
        \node[circle, draw=black,scale=0.5](5) at (3,0) {};
        \node[circle, draw=black,scale=0.5](6) at (4,0) {};
        \node[]() at (-1,0.5) {$1$};
        \node[]() at (0,0.5) {$2$};
        \node[]() at (1,0.5) {$2$};
        \node[]() at (2,0.5) {$2$};
        \node[]() at (3,0.5) {$2$};
        \node[]() at (4,0.5) {$1$};

    \path (1) edge (2);
    \path (2) edge (3);
    \path (3) edge (4);
    \path (4) edge (5);
    \path (5) edge (6);
    \end{tikzpicture}
\end{center}
    \end{minipage}
    & 
        \begin{minipage}{0.35\textwidth}
                \begin{center}
    \begin{tikzpicture}
        \node[circle, fill=black,scale=0.5](1) at (-1,0) {};
        \node[circle, draw=black,scale=0.5](2) at (0,0) {};
        \node[circle, fill=black,scale=0.5](3) at (1,0) {};
        \node[circle, fill=black,scale=0.5](4) at (2,0) {};
        \node[circle, draw=black,scale=0.5](5) at (3,0) {};
        \node[circle, draw=black,scale=0.5](6) at (4,0) {};
        \node[]() at (-1,0.5) {$1$};
        \node[]() at (0,0.5) {$2$};
        \node[]() at (1,0.5) {$2$};
        \node[]() at (2,0.5) {$2$};
        \node[]() at (3,0.5) {$1$};
        \node[]() at (4,0.5) {$0$};

    \path (1) edge (2);
    \path (2) edge (3);
    \path (3) edge (4);
    \path (4) edge (5);
    \path (5) edge (6);
    \end{tikzpicture}
\end{center}
    \end{minipage}
    & 0&$L+\sum_{i=1}^3(E_i+A_i)$ \\ \hline 
    $Y_{5,3}$
    &    
        \begin{minipage}{0.35\textwidth}
                 \begin{center}
    \begin{tikzpicture}
    
        \node[circle, draw=black,scale=0.5](-1) at (-3,0) {};
        \node[circle, draw=black,scale=0.5](0) at (-2,0) {};
        \node[circle, fill=black,scale=0.5](1) at (-1,0) {};
        \node[circle, fill=black,scale=0.5](2) at (0,0) {};
        \node[circle, draw=black,scale=0.5](3) at (1,0.5) {};
        \node[circle, draw=black,scale=0.5](4) at (1,-0.5) {};
        
        \node[]() at (-3,0.5) {$1$};
        \node[]() at (-2,0.5) {$2$};
        \node[]() at (-1,0.5) {$2$};
        \node[]() at (0,0.5) {$2$};
        \node[]() at (1,1) {$1$};
        \node[]() at (1,0) {$1$};
        
    \path (-1) edge (0);
    \path (0) edge (1);
    \path (1) edge (2);
    \path (2) edge (3);
    \path (2) edge (4);
    \end{tikzpicture}
\end{center}
    \end{minipage}
    & 
        \begin{minipage}{0.35\textwidth}
                     \begin{center}
    \begin{tikzpicture}
    
        \node[circle, draw=black,scale=0.5](-1) at (-3,0) {};
        \node[circle, draw=black,scale=0.5](0) at (-2,0) {};
        \node[circle, fill=black,scale=0.5](1) at (-1,0) {};
        \node[circle, fill=black,scale=0.5](2) at (0,0) {};
        \node[circle, draw=black,scale=0.5](3) at (1,0.5) {};
        \node[circle, draw=black,scale=0.5](4) at (1,-0.5) {};
        
        \node[]() at (-3,0.5) {$0$};
        \node[]() at (-2,0.5) {$1$};
        \node[]() at (-1,0.5) {$2$};
        \node[]() at (0,0.5) {$2$};
        \node[]() at (1,1) {$1$};
        \node[]() at (1,0) {$1$};
        
    \path (-1) edge (0);
    \path (0) edge (1);
    \path (1) edge (2);
    \path (2) edge (3);
    \path (2) edge (4);
    \end{tikzpicture}
\end{center}
    \end{minipage}
    & 1&{\footnotesize$L+\sum_{i=1}^3E_i+A_1+A_2$} \\ \hline
    $Y_{5,4}$
    & 
        \begin{minipage}{0.35\textwidth}
             \begin{center}
    \begin{tikzpicture}
        \node[circle, fill=black,scale=0.5](0) at (-2,0) {};
        \node[circle, fill=black,scale=0.5](1) at (-1,0) {};
        \node[circle, fill=black,scale=0.5](2) at (0,0) {};
        \node[circle, draw=black,scale=0.5](3) at (1,0.5) {};
        \node[circle, fill=black,scale=0.5](4) at (1,-0.5) {};
        
        \node[]() at (-2,0.5) {$2$};
        \node[]() at (-1,0.5) {$4$};
        \node[]() at (0,0.5) {$6$};
        \node[]() at (1,1) {$5$};
        \node[]() at (1,0) {$3$};
        
    \path (0) edge (1);
    \path (1) edge (2);
    \path (2) edge (3);
    \path (2) edge (4);
    \end{tikzpicture}
\end{center}
    \end{minipage}
    &   
        \begin{minipage}{0.35\textwidth}
                 \begin{center}
    \begin{tikzpicture}
        \node[circle, fill=black,scale=0.5](0) at (-2,0) {};
        \node[circle, fill=black,scale=0.5](1) at (-1,0) {};
        \node[circle, fill=black,scale=0.5](2) at (0,0) {};
        \node[circle, draw=black,scale=0.5](3) at (1,0.5) {};
        \node[circle, fill=black,scale=0.5](4) at (1,-0.5) {};
        
        \node[]() at (-2,0.5) {$1$};
        \node[]() at (-1,0.5) {$3$};
        \node[]() at (0,0.5) {$2$};
        \node[]() at (1,1) {$1$};
        \node[]() at (1,0) {$1$};
        
    \path (0) edge (1);
    \path (1) edge (2);
    \path (2) edge (3);
    \path (2) edge (4);
    \end{tikzpicture}
\end{center}
    
    \end{minipage}
    & 1&$L+H+\sum_{i=1}^4A_i$\\ \hline  
    $Y_{4,1}$ 
    & 
        \begin{minipage}{0.35\textwidth}
         \begin{center}
    \begin{tikzpicture}
    
        \node[circle, draw=black,scale=0.5](-2) at (-4,0) {};
        \node[circle, draw=black,scale=0.5](-1) at (-3,0) {};
        \node[circle, fill=black,scale=0.5](0) at (-2,0) {};
        \node[circle, fill=black,scale=0.5](1) at (-1,0) {};
        \node[circle, fill=black,scale=0.5](2) at (0,0) {};
        \node[circle, draw=black,scale=0.5](3) at (1,0.5) {};
        \node[circle, fill=black,scale=0.5](4) at (1,-0.5) {};
        
        \node[]() at (-4,0.5) {$0$};
        \node[]() at (-3,0.5) {$1$};
        \node[]() at (-2,0.5) {$2$};
        \node[]() at (-1,0.5) {$3$};
        \node[]() at (0,0.5) {$4$};
        \node[]() at (1,1) {$3$};
        \node[]() at (1,0) {$2$};
        
    \path (-2) edge (-1);
    \path (-1) edge (0);
    \path (0) edge (1);
    \path (1) edge (2);
    \path (2) edge (3);
    \path (2) edge (4);
    \end{tikzpicture}
\end{center}
    \end{minipage}
    & 
        \begin{minipage}{0.35\textwidth}
             \begin{center}
    \begin{tikzpicture}
    
        \node[circle, draw=black,scale=0.5](-2) at (-4,0) {};
        \node[circle, draw=black,scale=0.5](-1) at (-3,0) {};
        \node[circle, fill=black,scale=0.5](0) at (-2,0) {};
        \node[circle, fill=black,scale=0.5](1) at (-1,0) {};
        \node[circle, fill=black,scale=0.5](2) at (0,0) {};
        \node[circle, draw=black,scale=0.5](3) at (1,0.5) {};
        \node[circle, fill=black,scale=0.5](4) at (1,-0.5) {};
        
        \node[]() at (-4,0.5) {$0$};
        \node[]() at (-3,0.5) {$1$};
        \node[]() at (-2,0.5) {$2$};
        \node[]() at (-1,0.5) {$2$};
        \node[]() at (0,0.5) {$2$};
        \node[]() at (1,1) {$1$};
        \node[]() at (1,0) {$1$};
        
    \path (-2) edge (-1);
    \path (-1) edge (0);
    \path (0) edge (1);
    \path (1) edge (2);
    \path (2) edge (3);
    \path (2) edge (4);
    \end{tikzpicture}
\end{center}
    \end{minipage}
    & 1&$L+E_1+\sum_{i=1}^4A_i$ \\ \hline
    $Y_{4,2}$ 
    &  
        \begin{minipage}{0.35\textwidth}
            \begin{center}
    \begin{tikzpicture}
        \node[circle, draw=black,scale=0.5](1) at (-1,0) {};
        \node[circle, fill=black,scale=0.5](2) at (0,0) {};
        \node[circle, fill=black,scale=0.5](3) at (1,0) {};
        \node[circle, fill=black,scale=0.5](3*) at (1,1) {};
        \node[circle, fill=black,scale=0.5](4) at (2,0) {};
        \node[circle, draw=black,scale=0.5](5) at (3,0) {};
        \node[]() at (-1,0.5) {$2$};
        \node[]() at (0,0.5) {$3$};
        \node[]() at (1.2,0.5) {$4$};
        \node[]() at (1,1.5) {$2$};
        \node[]() at (2,0.5) {$3$};
        \node[]() at (3,0.5) {$2$};
        
    \path (1) edge (2);
    \path (2) edge (3);
    \path (3) edge (3*);
    \path (3) edge (4);
    \path (4) edge (5);
    \end{tikzpicture}
\end{center}
    \end{minipage}
    & 
        \begin{minipage}{0.35\textwidth}
      \begin{center}
    \begin{tikzpicture}
        \node[circle, draw=black,scale=0.5](1) at (-1,0) {};
        \node[circle, fill=black,scale=0.5](2) at (0,0) {};
        \node[circle, fill=black,scale=0.5](3) at (1,0) {};
        \node[circle, fill=black,scale=0.5](3*) at (1,1) {};
        \node[circle, fill=black,scale=0.5](4) at (2,0) {};
        \node[circle, draw=black,scale=0.5](5) at (3,0) {};
        \node[]() at (-1,0.5) {$0$};
        \node[]() at (0,0.5) {$1$};
        \node[]() at (1.2,0.5) {$2$};
        \node[]() at (1,1.5) {$1$};
        \node[]() at (2,0.5) {$1$};
        \node[]() at (3,0.5) {$0$};
        
    \path (1) edge (2);
    \path (2) edge (3);
    \path (3) edge (3*);
    \path (3) edge (4);
    \path (4) edge (5);
    \end{tikzpicture}
\end{center}
    \end{minipage}
    & $\frac{3}{2}$ &
    $\frac{1}{4}D_K+\frac{3}{4}(L_1+L_2)$
    \\ \hline
    $Y_{4,3}$
    &
        \begin{minipage}{0.35\textwidth}
             \begin{center}
    \begin{tikzpicture}
    
        \node[circle, fill=black,scale=0.5](-2) at (-4,0) {};
        \node[circle, draw=black,scale=0.5](-1) at (-3,0) {};
        \node[circle, fill=black,scale=0.5](0) at (-2,0) {};
        \node[circle, fill=black,scale=0.5](1) at (-1,0) {};
        \node[circle, fill=black,scale=0.5](2) at (0,0) {};
        \node[circle, draw=black,scale=0.5](3) at (1,0.5) {};
        \node[circle, draw=black,scale=0.5](4) at (1,-0.5) {};
        
        \node[]() at (-4,0.5) {$1$};
        \node[]() at (-3,0.5) {$2$};
        \node[]() at (-2,0.5) {$2$};
        \node[]() at (-1,0.5) {$2$};
        \node[]() at (0,0.5) {$2$};
        \node[]() at (1,1) {$1$};
        \node[]() at (1,0) {$1$};
        
    \path (-2) edge (-1);
    \path (-1) edge (0);
    \path (0) edge (1);
    \path (1) edge (2);
    \path (2) edge (3);
    \path (2) edge (4);
    \end{tikzpicture}
\end{center}
    \end{minipage}
    & 
        \begin{minipage}{0.35\textwidth}
     \begin{center}
    \begin{tikzpicture}
    
        \node[circle, fill=black,scale=0.5](-2) at (-4,0) {};
        \node[circle, draw=black,scale=0.5](-1) at (-3,0) {};
        \node[circle, fill=black,scale=0.5](0) at (-2,0) {};
        \node[circle, fill=black,scale=0.5](1) at (-1,0) {};
        \node[circle, fill=black,scale=0.5](2) at (0,0) {};
        \node[circle, draw=black,scale=0.5](3) at (1,0.5) {};
        \node[circle, draw=black,scale=0.5](4) at (1,-0.5) {};
        
        \node[]() at (-4,0.5) {$0$};
        \node[]() at (-3,0.5) {$0$};
        \node[]() at (-2,0.5) {$1$};
        \node[]() at (-1,0.5) {$2$};
        \node[]() at (0,0.5) {$2$};
        \node[]() at (1,1) {$1$};
        \node[]() at (1,0) {$1$};
        
    \path (-2) edge (-1);
    \path (-1) edge (0);
    \path (0) edge (1);
    \path (1) edge (2);
    \path (2) edge (3);
    \path (2) edge (4);
    \end{tikzpicture}
\end{center}
    \end{minipage}
    & 1&$L+E_1+\sum_{i=1}^4A_i$\\

    \\ \hline

    \end{tabular}
\end{table}

    \begin{table}
    \begin{tabular}{ | l | l | l | l | p{2cm}|}\\
    \hline
    Surface & $D_K$,Decomposition of $-K_Y$ & Slave divisor & $\gamma(X)$ &  $D\equiv -K_Y$ \\
    \hline
    $Y_{4,4}$ 
    & 
        \begin{minipage}{0.35\textwidth}
                \begin{center}
    \begin{tikzpicture}
        \node[circle, fill=black,scale=0.5](0) at (-2,0) {};
        \node[circle, draw=black,scale=0.5](1) at (-1,0) {};
        \node[circle, fill=black,scale=0.5](2) at (0,0) {};
        \node[circle, fill=black,scale=0.5](3) at (1,0) {};
        \node[circle, fill=black,scale=0.5](4) at (2,0) {};
        \node[circle, draw=black,scale=0.5](5) at (3,0) {};
        \node[circle, fill=black,scale=0.5](6) at (4,0) {};
        
        \node[]() at (-2,0.5) {$1$};
        \node[]() at (-1,0.5) {$2$};
        \node[]() at (0,0.5) {$2$};
        \node[]() at (1,0.5) {$2$};
        \node[]() at (2,0.5) {$2$};
        \node[]() at (3,0.5) {$2$};
        \node[]() at (4,0.5) {$1$};

    \path (0) edge (1);
    \path (1) edge (2);
    \path (2) edge (3);
    \path (3) edge (4);
    \path (4) edge (5);
    \path (5) edge (6);
    \end{tikzpicture}
\end{center}
    \end{minipage}
    & 
        \begin{minipage}{0.35\textwidth}
                    \begin{center}
    \begin{tikzpicture}
        \node[circle, fill=black,scale=0.5](0) at (-2,0) {};
        \node[circle, draw=black,scale=0.5](1) at (-1,0) {};
        \node[circle, fill=black,scale=0.5](2) at (0,0) {};
        \node[circle, fill=black,scale=0.5](3) at (1,0) {};
        \node[circle, fill=black,scale=0.5](4) at (2,0) {};
        \node[circle, draw=black,scale=0.5](5) at (3,0) {};
        \node[circle, fill=black,scale=0.5](6) at (4,0) {};
        
        \node[]() at (-2,0.5) {$0$};
        \node[]() at (-1,0.5) {$0$};
        \node[]() at (0,0.5) {$1$};
        \node[]() at (1,0.5) {$2$};
        \node[]() at (2,0.5) {$2$};
        \node[]() at (3,0.5) {$2$};
        \node[]() at (4,0.5) {$1$};

    \path (0) edge (1);
    \path (1) edge (2);
    \path (2) edge (3);
    \path (3) edge (4);
    \path (4) edge (5);
    \path (5) edge (6);
    \end{tikzpicture}
\end{center}
    \end{minipage}
    & 0&{\footnotesize $L+E_1+E_2+\sum_{i=1}^5A_i$} \\ \hline
    $Y_{4,5}$ 
    & 
        \begin{minipage}{0.3\textwidth}
         \begin{center}
    \begin{tikzpicture}
    
        \node[circle, draw=black,scale=0.5](-2) at (-3,0.5) {};
        \node[circle, draw=black,scale=0.5](-1) at (-3,-0.5) {};
        \node[circle, fill=black,scale=0.5](0) at (-2,0) {};
        \node[circle, fill=black,scale=0.5](1) at (-1,0) {};
        \node[circle, fill=black,scale=0.5](2) at (0,0) {};
        \node[circle, draw=black,scale=0.5](3) at (1,0.5) {};
        \node[circle, draw=black,scale=0.5](4) at (1,-0.5) {};
        
        \node[]() at (-3,1) {$1$};
        \node[]() at (-3,0) {$1$};
        \node[]() at (-2,0.5) {$2$};
        \node[]() at (-1,0.5) {$2$};
        \node[]() at (0,0.5) {$2$};
        \node[]() at (1,1) {$1$};
        \node[]() at (1,0) {$1$};
        
    \path (-2) edge (0);
    \path (-1) edge (0);
    \path (0) edge (1);
    \path (1) edge (2);
    \path (2) edge (3);
    \path (2) edge (4);
    \end{tikzpicture}
\end{center}
    \end{minipage}
    & 
        \begin{minipage}{0.3\textwidth}
             \begin{center}
    \begin{tikzpicture}
    
        \node[circle, draw=black,scale=0.5](-2) at (-3,0.5) {};
        \node[circle, draw=black,scale=0.5](-1) at (-3,-0.5) {};
        \node[circle, fill=black,scale=0.5](0) at (-2,0) {};
        \node[circle, fill=black,scale=0.5](1) at (-1,0) {};
        \node[circle, fill=black,scale=0.5](2) at (0,0) {};
        \node[circle, draw=black,scale=0.5](3) at (1,0.5) {};
        \node[circle, draw=black,scale=0.5](4) at (1,-0.5) {};
        
        \node[]() at (-3,1) {$0$};
        \node[]() at (-3,0) {$0$};
        \node[]() at (-2,0.5) {$1$};
        \node[]() at (-1,0.5) {$2$};
        \node[]() at (0,0.5) {$2$};
        \node[]() at (1,1) {$1$};
        \node[]() at (1,0) {$1$};
        
    \path (-2) edge (0);
    \path (-1) edge (0);
    \path (0) edge (1);
    \path (1) edge (2);
    \path (2) edge (3);
    \path (2) edge (4);
    \end{tikzpicture}
\end{center}
    \end{minipage}
    & 2&$L+E_2+E_4+\sum_{i=1}^3A_i$ \\ \hline
    $Y_{4,6}$ 
    & 
        \begin{minipage}{0.3\textwidth}
    \begin{center}
    \begin{tikzpicture}
        \node[circle, draw=black,scale=0.5](1) at (-1,0) {};
        \node[circle, fill=black,scale=0.5](2) at (0,0) {};
        \node[circle, fill=black,scale=0.5](3) at (1,0) {};
        \node[circle, fill=black,scale=0.5](3*) at (1,1) {};
        \node[circle, fill=black,scale=0.5](4) at (2,0) {};
        \node[circle, fill=black,scale=0.5](5) at (3,0) {};
        \node[]() at (-1,0.5) {$4$};
        \node[]() at (0,0.5) {$5$};
        \node[]() at (1.2,0.5) {$6$};
        \node[]() at (1,1.5) {$3$};
        \node[]() at (2,0.5) {$4$};
        \node[]() at (3,0.5) {$2$};
        
    \path (1) edge (2);
    \path (2) edge (3);
    \path (3) edge (3*);
    \path (3) edge (4);
    \path (4) edge (5);
    \end{tikzpicture}
\end{center}
    \end{minipage}
    & 
        \begin{minipage}{0.3\textwidth}
        \begin{center}
    \begin{tikzpicture}
        \node[circle, draw=black,scale=0.5](1) at (-1,0) {};
        \node[circle, fill=black,scale=0.5](2) at (0,0) {};
        \node[circle, fill=black,scale=0.5](3) at (1,0) {};
        \node[circle, fill=black,scale=0.5](3*) at (1,1) {};
        \node[circle, fill=black,scale=0.5](4) at (2,0) {};
        \node[circle, fill=black,scale=0.5](5) at (3,0) {};
        \node[]() at (-1,0.5) {$1$};
        \node[]() at (0,0.5) {$2$};
        \node[]() at (1.2,0.5) {$3$};
        \node[]() at (1,1.5) {$1.5$};
        \node[]() at (2,0.5) {$2$};
        \node[]() at (3,0.5) {$1$};
        
    \path (1) edge (2);
    \path (2) edge (3);
    \path (3) edge (3*);
    \path (3) edge (4);
    \path (4) edge (5);
    \end{tikzpicture}
\end{center}
    \end{minipage}
    & $\frac{3}{2}$ &$\frac{1}{6}D_K+\frac{5}{6}\mathcal{C}$\\ \hline
    
    $Y_{3,1}$ 
    & 
        \begin{minipage}{0.3\textwidth}
            \begin{center}
    \begin{tikzpicture}
        \node[circle, draw=black,scale=0.5](-1) at (-3,0) {};
        \node[circle, draw=black,scale=0.5](0) at (-2,0) {};
        \node[circle, fill=black,scale=0.5](1) at (-1,0) {};
        \node[circle, fill=black,scale=0.5](2) at (0,0) {};
        \node[circle, fill=black,scale=0.5](3) at (1,0) {};
        \node[circle, fill=black,scale=0.5](3*) at (1,1) {};
        \node[circle, fill=black,scale=0.5](4) at (2,0) {};
        \node[circle, draw=black,scale=0.5](5) at (3,0) {};
        \node[]() at (-3,0.5) {$1$};
        \node[]() at (-2,0.5) {$2$};
        \node[]() at (-1,0.5) {$2$};
        \node[]() at (0,0.5) {$2$};
        \node[]() at (1.2,0.5) {$2$};
        \node[]() at (1,1.5) {$1$};
        \node[]() at (2,0.5) {$1$};
        \node[]() at (3,0.5) {$0$};
        
    \path (-1) edge (0);
    \path (0) edge (1);
    \path (1) edge (2);
    \path (2) edge (3);
    \path (3) edge (3*);
    \path (3) edge (4);
    \path (4) edge (5);
    \end{tikzpicture}
\end{center}
    \end{minipage}
    & 
        \begin{minipage}{0.3\textwidth}
                \begin{center}
    \begin{tikzpicture}
        \node[circle, draw=black,scale=0.5](-1) at (-3,0) {};
        \node[circle, draw=black,scale=0.5](0) at (-2,0) {};
        \node[circle, fill=black,scale=0.5](1) at (-1,0) {};
        \node[circle, fill=black,scale=0.5](2) at (0,0) {};
        \node[circle, fill=black,scale=0.5](3) at (1,0) {};
        \node[circle, fill=black,scale=0.5](3*) at (1,1) {};
        \node[circle, fill=black,scale=0.5](4) at (2,0) {};
        \node[circle, draw=black,scale=0.5](5) at (3,0) {};
        \node[]() at (-3,0.5) {$0$};
        \node[]() at (-2,0.5) {$0$};
        \node[]() at (-1,0.5) {$1$};
        \node[]() at (0,0.5) {$2$};
        \node[]() at (1.2,0.5) {$2$};
        \node[]() at (1,1.5) {$1$};
        \node[]() at (2,0.5) {$1$};
        \node[]() at (3,0.5) {$0$};
        
    \path (-1) edge (0);
    \path (0) edge (1);
    \path (1) edge (2);
    \path (2) edge (3);
    \path (3) edge (3*);
    \path (3) edge (4);
    \path (4) edge (5);
    \end{tikzpicture}
\end{center}
    \end{minipage}
    & $\frac{3}{2}$&$\frac{1}{2}D_K+\frac{1}{2}(L+E_1)$ \\ \hline
    $Y_{3,2}$ 
    & 
        \begin{minipage}{0.3\textwidth}
             \begin{center}
    \begin{tikzpicture}
    
        \node[circle, draw=black,scale=0.5](-2) at (-3,0.5) {};
        \node[circle, draw=black,scale=0.5](-1) at (-3,-0.5) {};
        \node[circle, fill=black,scale=0.5](0) at (-2,0) {};
        \node[circle, fill=black,scale=0.5](0*) at (-1,0) {};
        \node[circle, fill=black,scale=0.5](1) at (0,0) {};
        \node[circle, fill=black,scale=0.5](2) at (1,0) {};
        \node[circle, fill=black,scale=0.5](3) at (2,0.5) {};
        \node[circle, draw=black,scale=0.5](4) at (2,-0.5) {};
        
        \node[]() at (-3,1) {$1$};
        \node[]() at (-3,0) {$1$};
        \node[]() at (-2,0.5) {$2$};
        \node[]() at (-1,0.5) {$2$};
        \node[]() at (0,0.5) {$2$};
        \node[]() at (2,1) {$1$};
        \node[]() at (2,1) {$1$};
        \node[]() at (2,0) {$1$};
        
    \path (-2) edge (0);
    \path (-1) edge (0);
    \path (0) edge (0*);
    \path (0*) edge (1);
    \path (1) edge (2);
    \path (2) edge (3);
    \path (2) edge (4);
    \end{tikzpicture}
\end{center}
    \end{minipage}
    & 
        \begin{minipage}{0.3\textwidth}
                 \begin{center}
    \begin{tikzpicture}
    
        \node[circle, draw=black,scale=0.5](-2) at (-3,0.5) {};
        \node[circle, draw=black,scale=0.5](-1) at (-3,-0.5) {};
        \node[circle, fill=black,scale=0.5](0) at (-2,0) {};
        \node[circle, fill=black,scale=0.5](0*) at (-1,0) {};
        \node[circle, fill=black,scale=0.5](1) at (0,0) {};
        \node[circle, fill=black,scale=0.5](2) at (1,0) {};
        \node[circle, fill=black,scale=0.5](3) at (2,0.5) {};
        \node[circle, draw=black,scale=0.5](4) at (2,-0.5) {};
        
        \node[]() at (-3,1) {$0$};
        \node[]() at (-3,0) {$0$};
        \node[]() at (-2,0.5) {$1$};
        \node[]() at (-1,0.5) {$2$};
        \node[]() at (0,0.5) {$2$};
        \node[]() at (1,0.5) {$2$};
        \node[]() at (2,1) {$1$};
        \node[]() at (2,0) {$1$};
        
    \path (-2) edge (0);
    \path (-1) edge (0);
    \path (0) edge (0*);
    \path (0*) edge (1);
    \path (1) edge (2);
    \path (2) edge (3);
    \path (2) edge (4);
    \end{tikzpicture}
\end{center}
    \end{minipage}
    & 2& $L+E_1+\sum_{i=1}^5A_i$ \\ \hline
    $Y_{3,3}$ 
    & 
        \begin{minipage}{0.3\textwidth}
                 \begin{center}
    \begin{tikzpicture}
    
        \node[circle, fill=black,scale=0.5](-3) at (-5,0) {};
        \node[circle, draw=black,scale=0.5](-2) at (-4,0) {};
        \node[circle, fill=black,scale=0.5](-1) at (-3,0) {};
        \node[circle, fill=black,scale=0.5](0) at (-2,0) {};
        \node[circle, fill=black,scale=0.5](1) at (-1,0) {};
        \node[circle, fill=black,scale=0.5](2) at (0,0) {};
        \node[circle, draw=black,scale=0.5](3) at (1,0.5) {};
        \node[circle, fill=black,scale=0.5](4) at (1,-0.5) {};
        
        \node[]() at (-5,0.5) {$1$};
        \node[]() at (-4,0.5) {$2$};
        \node[]() at (-3,0.5) {$2$};
        \node[]() at (-2,0.5) {$2$};
        \node[]() at (-1,0.5) {$2$};
        \node[]() at (0,0.5) {$2$};
        \node[]() at (1,1) {$1$};
        \node[]() at (1,0) {$1$};
        
    \path (-3) edge (-2);
    \path (-2) edge (-1);
    \path (-1) edge (0);
    \path (0) edge (1);
    \path (1) edge (2);
    \path (2) edge (3);
    \path (2) edge (4);
    \end{tikzpicture}
\end{center}
    \end{minipage}
    & 
        \begin{minipage}{0.3\textwidth}
                     \begin{center}
    \begin{tikzpicture}
    
        \node[circle, fill=black,scale=0.5](-3) at (-5,0) {};
        \node[circle, draw=black,scale=0.5](-2) at (-4,0) {};
        \node[circle, fill=black,scale=0.5](-1) at (-3,0) {};
        \node[circle, fill=black,scale=0.5](0) at (-2,0) {};
        \node[circle, fill=black,scale=0.5](1) at (-1,0) {};
        \node[circle, fill=black,scale=0.5](2) at (0,0) {};
        \node[circle, draw=black,scale=0.5](3) at (1,0.5) {};
        \node[circle, fill=black,scale=0.5](4) at (1,-0.5) {};
        
        \node[]() at (-5,0.5) {$0$};
        \node[]() at (-4,0.5) {$0$};
        \node[]() at (-3,0.5) {$1$};
        \node[]() at (-2,0.5) {$2$};
        \node[]() at (-1,0.5) {$2$};
        \node[]() at (0,0.5) {$2$};
        \node[]() at (1,1) {$1$};
        \node[]() at (1,0) {$1$};
        
    \path (-3) edge (-2);
    \path (-2) edge (-1);
    \path (-1) edge (0);
    \path (0) edge (1);
    \path (1) edge (2);
    \path (2) edge (3);
    \path (2) edge (4);
    \end{tikzpicture}
\end{center}
    \end{minipage}
    & 1&$L+E_1+\sum_{i=1}^6A_i$\\ \hline
    $Y_{3,4}$ 
    & 
        \begin{minipage}{0.3\textwidth}
        \begin{center}
    \begin{tikzpicture}
        \node[circle, draw=black,scale=0.5](0) at (-2,0) {};
        \node[circle, fill=black,scale=0.5](1) at (-1,0) {};
        \node[circle, fill=black,scale=0.5](2) at (0,0) {};
        \node[circle, fill=black,scale=0.5](3) at (1,0) {};
        \node[circle, fill=black,scale=0.5](3*) at (1,1) {};
        \node[circle, fill=black,scale=0.5](4) at (2,0) {};
        \node[circle, fill=black,scale=0.5](5) at (3,0) {};
        \node[]() at (-2,0.5) {$3$};
        \node[]() at (-1,0.5) {$4$};
        \node[]() at (0,0.5) {$5$};
        \node[]() at (1.2,0.5) {$6$};
        \node[]() at (1,1.5) {$3$};
        \node[]() at (2,0.5) {$4$};
        \node[]() at (3,0.5) {$2$};
        
    \path (0) edge (1);
    \path (1) edge (2);
    \path (2) edge (3);
    \path (3) edge (3*);
    \path (3) edge (4);
    \path (4) edge (5);
    \end{tikzpicture}
\end{center}
    \end{minipage}
    & 
        \begin{minipage}{0.3\textwidth}
            \begin{center}
    \begin{tikzpicture}
        \node[circle, draw=black,scale=0.5](0) at (-2,0) {};
        \node[circle, fill=black,scale=0.5](1) at (-1,0) {};
        \node[circle, fill=black,scale=0.5](2) at (0,0) {};
        \node[circle, fill=black,scale=0.5](3) at (1,0) {};
        \node[circle, fill=black,scale=0.5](3*) at (1,1) {};
        \node[circle, fill=black,scale=0.5](4) at (2,0) {};
        \node[circle, fill=black,scale=0.5](5) at (3,0) {};
        \node[]() at (-2,0.5) {$1$};
        \node[]() at (-1,0.5) {$4$};
        \node[]() at (0,0.5) {$5$};
        \node[]() at (1.2,0.5) {$6$};
        \node[]() at (1,1.5) {$3$};
        \node[]() at (2,0.5) {$4$};
        \node[]() at (3,0.5) {$2$};
        
    \path (0) edge (1);
    \path (1) edge (2);
    \path (2) edge (3);
    \path (3) edge (3*);
    \path (3) edge (4);
    \path (4) edge (5);
    \end{tikzpicture}
\end{center}
    \end{minipage}
    & $\frac{5}{3}$ & $\frac{1}{6}D_K+\frac{5}{6}\mathcal{C}$\\ \hline
    
    $Y_{2,1}$ 
    & 
        \begin{minipage}{0.3\textwidth}
                    \begin{center}
    \begin{tikzpicture}
        \node[circle, fill=black,scale=0.5](-2) at (-4*3/4,0) {};
        \node[circle, draw=black,scale=0.5](-1) at (-3*3/4,0) {};
        \node[circle, fill=black,scale=0.5](0) at (-2*3/4,0) {};
        \node[circle, fill=black,scale=0.5](1) at (-1*3/4,0) {};
        \node[circle, fill=black,scale=0.5](2) at (0*3/4,0) {};
        \node[circle, fill=black,scale=0.5](3) at (1*3/4,0) {};
        \node[circle, fill=black,scale=0.5](3*) at (1*3/4,1) {};
        \node[circle, fill=black,scale=0.5](4) at (2*3/4,0) {};
        \node[circle, draw=black,scale=0.5](5) at (3*3/4,0) {};
        \node[]() at (-4*3/4,0.5) {$1$};
        \node[]() at (-3*3/4,0.5) {$2$};
        \node[]() at (-2*3/4,0.5) {$2$};
        \node[]() at (-1*3/4,0.5) {$2$};
        \node[]() at (0*3/4,0.5) {$2$};
        \node[]() at (1.2*3/4,0.5) {$2$};
        \node[]() at (1*3/4,1.5) {$1$};
        \node[]() at (2*3/4,0.5) {$1$};
        \node[]() at (3*3/4,0.5) {$0$};
        
    \path (-2) edge (-1);
    \path (-1) edge (0);
    \path (0) edge (1);
    \path (1) edge (2);
    \path (2) edge (3);
    \path (3) edge (3*);
    \path (3) edge (4);
    \path (4) edge (5);
    \end{tikzpicture}
\end{center}
    \end{minipage}
    & 
        \begin{minipage}{0.3\textwidth}
                        \begin{center}
    \begin{tikzpicture}
        \node[circle, fill=black,scale=0.5](-2) at (-4*3/4,0) {};
        \node[circle, draw=black,scale=0.5](-1) at (-3*3/4,0) {};
        \node[circle, fill=black,scale=0.5](0) at (-2*3/4,0) {};
        \node[circle, fill=black,scale=0.5](1) at (-1*3/4,0) {};
        \node[circle, fill=black,scale=0.5](2) at (0*3/4,0) {};
        \node[circle, fill=black,scale=0.5](3) at (1*3/4,0) {};
        \node[circle, fill=black,scale=0.5](3*) at (1*3/4,1) {};
        \node[circle, fill=black,scale=0.5](4) at (2*3/4,0) {};
        \node[circle, draw=black,scale=0.5](5) at (3*3/4,0) {};
        \node[]() at (-4*3/4,0.5) {$0$};
        \node[]() at (-3*3/4,0.5) {$0$};
        \node[]() at (-2*3/4,0.5) {$1$};
        \node[]() at (-1*3/4,0.5) {$2$};
        \node[]() at (0*3/4,0.5) {$2$};
        \node[]() at (1.2*3/4,0.5) {$2$};
        \node[]() at (1*3/4,1.5) {$1$};
        \node[]() at (2*3/4,0.5) {$1$};
        \node[]() at (3*3/4,0.5) {$0$};
        
    \path (-2) edge (-1);
    \path (-1) edge (0);
    \path (0) edge (1);
    \path (1) edge (2);
    \path (2) edge (3);
    \path (3) edge (3*);
    \path (3) edge (4);
    \path (4) edge (5);
    \end{tikzpicture}
\end{center}
    \end{minipage}
    & $\frac{3}{2}$ &$\frac{1}{2}D_K+\frac{1}{2}\mathcal{C}$\\ \hline
    $Y_{2,2}$ 
    & 
        \begin{minipage}{0.3\textwidth}
            \begin{center}
    \begin{tikzpicture}
        \node[circle, draw=black,scale=0.5](-2) at (-3,0.5) {};
        \node[circle, draw=black,scale=0.5](-1) at (-3,-0.5) {};
        \node[circle, fill=black,scale=0.5](0) at (-2,0) {};
        \node[circle, fill=black,scale=0.5](1) at (-1,0) {};
        \node[circle, fill=black,scale=0.5](2) at (0,0) {};
        \node[circle, fill=black,scale=0.5](3) at (1,0) {};
        \node[circle, fill=black,scale=0.5](3*) at (1,1) {};
        \node[circle, fill=black,scale=0.5](4) at (2,0) {};
        \node[circle, draw=black,scale=0.5](5) at (3,0) {};
        \node[]() at (-3,1) {$1$};
        \node[]() at (-3,0) {$1$};
        \node[]() at (-2,0.5) {$2$};
        \node[]() at (-1,0.5) {$2$};
        \node[]() at (0,0.5) {$2$};
        \node[]() at (1.2,0.5) {$2$};
        \node[]() at (1,1.5) {$1$};
        \node[]() at (2,0.5) {$1$};
        \node[]() at (3,0.5) {$0$};
        
    \path (-2) edge (0);
    \path (-1) edge (0);
    \path (0) edge (1);
    \path (1) edge (2);
    \path (2) edge (3);
    \path (3) edge (3*);
    \path (3) edge (4);
    \path (4) edge (5);
    \end{tikzpicture}
\end{center}
    \end{minipage}
    & 
        \begin{minipage}{0.3\textwidth}
                \begin{center}
    \begin{tikzpicture}
        \node[circle, draw=black,scale=0.5](-2) at (-3,0.5) {};
        \node[circle, draw=black,scale=0.5](-1) at (-3,-0.5) {};
        \node[circle, fill=black,scale=0.5](0) at (-2,0) {};
        \node[circle, fill=black,scale=0.5](1) at (-1,0) {};
        \node[circle, fill=black,scale=0.5](2) at (0,0) {};
        \node[circle, fill=black,scale=0.5](3) at (1,0) {};
        \node[circle, fill=black,scale=0.5](3*) at (1,1) {};
        \node[circle, fill=black,scale=0.5](4) at (2,0) {};
        \node[circle, draw=black,scale=0.5](5) at (3,0) {};
        \node[]() at (-3,1) {$0$};
        \node[]() at (-3,0) {$0$};
        \node[]() at (-2,0.5) {$1$};
        \node[]() at (-1,0.5) {$2$};
        \node[]() at (0,0.5) {$2$};
        \node[]() at (1.2,0.5) {$2$};
        \node[]() at (1,1.5) {$1$};
        \node[]() at (2,0.5) {$1$};
        \node[]() at (3,0.5) {$0$};
        
    \path (-2) edge (0);
    \path (-1) edge (0);
    \path (0) edge (1);
    \path (1) edge (2);
    \path (2) edge (3);
    \path (3) edge (3*);
    \path (3) edge (4);
    \path (4) edge (5);
    \end{tikzpicture}
\end{center}
    \end{minipage}
    & $\frac{5}{2}$&$\frac{1}{2}D_K+\frac{1}{2}\mathcal{C}$\\ \hline
    $Y_{2,3}$ 
    & 
        \begin{minipage}{0.3\textwidth}
                 \begin{center}
    \begin{tikzpicture}
    
        \node[circle, draw=black,scale=0.5](-2) at (-3,0.5) {};
        \node[circle, fill=black,scale=0.5](-1) at (-3,-0.5) {};
        \node[circle, fill=black,scale=0.5](0) at (-2,0) {};
        
        \node[circle, fill=black,scale=0.5](0*) at (-1,0) {};
        \node[circle, fill=black,scale=0.5](0**) at (0,0) {};
        \node[circle, fill=black,scale=0.5](1) at (1,0) {};
        \node[circle, fill=black,scale=0.5](2) at (2,0) {};
        \node[circle, draw=black,scale=0.5](3) at (3,0.5) {};
        \node[circle, fill=black,scale=0.5](4) at (3,-0.5) {};
        
        \node[]() at (-3,1) {$1$};
        \node[]() at (-3,0) {$1$};
        \node[]() at (-2,0.5) {$2$};
        \node[]() at (-1,0.5) {$2$};
        \node[]() at (0,0.5) {$2$};
        \node[]() at (1,0.5) {$2$};
        \node[]() at (2,0.5) {$2$};
        \node[]() at (3,1) {$1$};
        \node[]() at (3,0) {$1$};
        
    \path (-2) edge (0);
    \path (-1) edge (0);
    \path (0) edge (0*);
    \path (0*) edge (1);
    \path (1) edge (2);
    \path (2) edge (3);
    \path (2) edge (4);
    \end{tikzpicture}
\end{center}
    \end{minipage}
    & 
        \begin{minipage}{0.3\textwidth}
                     \begin{center}
    \begin{tikzpicture}
    
        \node[circle, draw=black,scale=0.5](-2) at (-3,0.5) {};
        \node[circle, fill=black,scale=0.5](-1) at (-3,-0.5) {};
        \node[circle, fill=black,scale=0.5](0) at (-2,0) {};
        
        \node[circle, fill=black,scale=0.5](0*) at (-1,0) {};
        \node[circle, fill=black,scale=0.5](0**) at (0,0) {};
        \node[circle, fill=black,scale=0.5](1) at (1,0) {};
        \node[circle, fill=black,scale=0.5](2) at (2,0) {};
        \node[circle, draw=black,scale=0.5](3) at (2.5,0.5) {};
        \node[circle, fill=black,scale=0.5](4) at (2.5,-0.5) {};
        
        \node[]() at (-3,1) {$0$};
        \node[]() at (-3,0) {$0.5$};
        \node[]() at (-2,0.5) {$1$};
        \node[]() at (-1,0.5) {$1$};
        \node[]() at (0,0.5) {$1$};
        \node[]() at (1,0.5) {$1$};
        \node[]() at (2,0.5) {$1$};
        \node[]() at (2.5,1) {$0$};
        \node[]() at (2.5,0) {$0.5$};
        
    \path (-2) edge (0);
    \path (-1) edge (0);
    \path (0) edge (0*);
    \path (0*) edge (1);
    \path (1) edge (2);
    \path (2) edge (3);
    \path (2) edge (4);
    \end{tikzpicture}
\end{center}
    \end{minipage}
    & $\frac{3}{2}$ & $\frac{1}{2}D_K+\frac{1}{2}\mathcal{C}$\\ \hline
    $Y_{2,4}$ 
    & 
        \begin{minipage}{0.3\textwidth}
                \begin{center}
    \begin{tikzpicture}
        \node[circle, draw=black,scale=0.5](-1) at (-3,0) {};
        \node[circle, fill=black,scale=0.5](0) at (-2,0) {};
        \node[circle, fill=black,scale=0.5](1) at (-1,0) {};
        \node[circle, fill=black,scale=0.5](2) at (0,0) {};
        \node[circle, fill=black,scale=0.5](3) at (1,0) {};
        \node[circle, fill=black,scale=0.5](3*) at (1,1) {};
        \node[circle, fill=black,scale=0.5](4) at (2,0) {};
        \node[circle, fill=black,scale=0.5](5) at (3,0) {};
        \node[]() at (-3,0.5) {$2$};
        \node[]() at (-2,0.5) {$3$};
        \node[]() at (-1,0.5) {$4$};
        \node[]() at (0,0.5) {$5$};
        \node[]() at (1.2,0.5) {$6$};
        \node[]() at (1,1.5) {$3$};
        \node[]() at (2,0.5) {$4$};
        \node[]() at (3,0.5) {$2$};
        
    \path (-1) edge (0);
    \path (0) edge (1);
    \path (1) edge (2);
    \path (2) edge (3);
    \path (3) edge (3*);
    \path (3) edge (4);
    \path (4) edge (5);
    \end{tikzpicture}
\end{center}
    \end{minipage}
    & 
        \begin{minipage}{0.3\textwidth}
                    \begin{center}
    \begin{tikzpicture}
        \node[circle, draw=black,scale=0.5](-1) at (-3,0) {};
        \node[circle, fill=black,scale=0.5](0) at (-2,0) {};
        \node[circle, fill=black,scale=0.5](1) at (-1,0) {};
        \node[circle, fill=black,scale=0.5](2) at (0,0) {};
        \node[circle, fill=black,scale=0.5](3) at (1,0) {};
        \node[circle, fill=black,scale=0.5](3*) at (1,1) {};
        \node[circle, fill=black,scale=0.5](4) at (2,0) {};
        \node[circle, fill=black,scale=0.5](5) at (3,0) {};
        \node[]() at (-3,0.5) {$1$};
        \node[]() at (-2,0.5) {$3$};
        \node[]() at (-1,0.5) {$4$};
        \node[]() at (0,0.5) {$5$};
        \node[]() at (1.2,0.5) {$6$};
        \node[]() at (1,1.5) {$3$};
        \node[]() at (2,0.5) {$4$};
        \node[]() at (3,0.5) {$2$};
        
    \path (-1) edge (0);
    \path (0) edge (1);
    \path (1) edge (2);
    \path (2) edge (3);
    \path (3) edge (3*);
    \path (3) edge (4);
    \path (4) edge (5);
    \end{tikzpicture}
\end{center}
    
    \end{minipage}
    &$\frac{11}{6}$&$\frac{1}{6}D_K+\frac{5}{6}\mathcal{C}$ \\ \hline
    \hline
    \end{tabular}

\end{table}

\end{theo}

%

\begin{proof}
How to use the table:
Denote  
\begin{equation*}
    -K_Y\equiv D = \sum d_iD_i + \sum e_jE_j + \sum a_kA_k ,\quad D_i^2\ge 0.
\end{equation*}
In the table $\mathcal{C}\in |-K_Y|$ is a non-singular curve.
So-called \textit{Slave divisor} $D_s$ is an integer (Except $Y_{4,6}$ and $Y_{2,3}$) combination of divisors from $D(Y)$. To find an upper boundary for $\sigma'(Y)$ we need either to find an upper boundary for $\sum e_j$ and apply lemma \ref{genus_zero} or to estimate $\sum d_i + \sum e_j$ straightforwardly. 
In the former case we need to find $(D,D_s)=(-K_Y,D_s)$, this gives the upper boundary for $\sum e_j$ and find $\sigma(X)=\frac{1}{2}(d+\mathrm{max}\lbrace\sum e_j\rbrace)$. To make the estimation sharp, one can take $D_i$ in the decomposition of $D$ such that $(D_i,-K_Y)=2$.
In the latter case we need to find $(D,D_s)$ again and use the fact that for any prime weil divisor $D_i$ such that $D_i^2\ge 0$ we have $(D_i,D_s)\ge 1$. From this follows 
\begin{equation*}
    (-K_Y,D_s) +\delta  \ge \sum d_i + \sum e_j ,
\end{equation*}
where $\delta$ appears because sometimes $(A_k,D_s)<0$ for some $A_k$ and also for some $E_j$ it may be that $(E_j,D_s)=0$, when usually $(E_j,D_s)=1$. To make the estimation sharp in this case one can take $D_i$ from the decomposition of $D$ such that $(D_i,D_s)=1$.
In the last column there is a divisor $D$ for which the estimation is sharp. 
Cases $Y_{3,4}$ and $Y_{2,4}$ require more exact approach.
Take $Y=Y_{3,4}$ and consider the contraction $\phi \colon Y \to Z=Y_{4,6}$ such that $\phi(E)=p$, where $E$ is the unique $(-1)$-curve on $Y$ and $p\in \widetilde{E}=\phi_*(A)$, where $A$ is a $(-2)$-curve on $Y$ such that $(A,E)=1$. it is clear that $\widetilde{E}$ is a $(-1)$-curve on $X$.
Suppose 
\begin{equation*}
    -K_Y\equiv D=\sum d_iD_i + \sum a_jA_j + aA+eE, \quad d_i,a_j,a,e \in \mathbb{Q} \cap (0,1]
\end{equation*}
Assume that $\sum d_i + e >\frac{4}{3}$
Then 
\begin{equation*}
    -K_Y\equiv \widetilde{D}=\sum d_i\widetilde{D_i}+\sum a_jA_j + a\widetilde{E}
\end{equation*}
where $\widetilde{D_i}=\phi_*(D_i)$ are prime Weil divisors and 
\begin{equation*}
    \widetilde{D_i}^2\ge D_i^2\ge 0
\end{equation*}  
holds.
From the table it is known that for $X=Y_{4,6}$ we have
$\sum d_i + a \le \frac{3}{2}$
 and then we get
\begin{equation*}
    \frac{4}{3} < \sum d_i + e = \sum d_i+a+(e-a) \le \frac{3}{2} +(e-a),
\end{equation*} 
from this follows $(a-e)<\frac{1}{6}$.
But on the other hand, using $D_s$ for $Y$ we deduce that
\begin{equation*}
    \sum d_i+e \le 1 + 2(a-e) < \frac{4}{3},
\end{equation*}
which is a contradiction.
Analogously, for $Y=Y_{2,4}$ take
\begin{equation*}
    D=\sum d_iD_i+\sum a_jA_j+aA+eE,\quad E^2=-1, \quad (A,E)=1.
\end{equation*}
Suppose that $\sum d_i+e > \frac{7}{6}$ .
Consider the contraction $\psi \colon Y \to X=Y_{3,4}$, such that $\phi(E)=p\in \phi_*(A)$ and $\phi_*(A)^2=-1$.
Deduce as in the previous paragraph that 
\begin{equation*}
    (a-e)<\frac{4}{3}-\frac{7}{6}=\frac{1}{6}
\end{equation*} 
and use $D_s$ for $Y$ to prove that 
\begin{equation*}
    \sum d_i+e \le 1+(a-e)<\frac{7}{6},
\end{equation*} which is a contradiction.
\end{proof}

From the table it is clearly seen that if $D(Y)$ is a tree with one branch, $\gamma(X)=0$. If $D(Y)$ is a tree with a branch and two leaves, then $\gamma(X)=1$. Other cases are more complicated.

\end{spacing}
\end{document}